\documentclass[centering,11pt,reqno]{amsart}  

\usepackage[utf8]{inputenc}
\usepackage[T1]{fontenc}
\usepackage{amsmath,amsthm}
\usepackage{amsfonts,amssymb}
\usepackage[mathscr]{eucal}
\usepackage{url}
\usepackage{paralist}
\usepackage[colorlinks,cite color=blue,pagebackref=true,pdftex]{hyperref}
\usepackage[margin=2.422cm]{geometry}
\usepackage{tikz}
\usepackage{algpseudocode}

\usepackage{cleveref}

\newtheorem{Itheorem}{Theorem}
\newtheorem{theorem}{Theorem}[section]
\newtheorem*{theorem*}{Theorem}
\newtheorem*{lemma*}{Lemma}
\newtheorem*{prop*}{Proposition}
\newtheorem{lemma}[theorem]{Lemma}
\newtheorem{cor}[theorem]{Corollary}
\newtheorem{prop}[theorem]{Proposition}

\newtheorem{op}{Open Problem}

\newtheorem{innercustomconj}{Conjecture}
\newenvironment{customconj}[1]
{\renewcommand\theinnercustomconj{#1}\innercustomconj}
{\endinnercustomconj}

\theoremstyle{definition}

\newtheorem{remark}[theorem]{Remark}

\newtheorem{examplex}[theorem]{Example}
\newenvironment{example}
  {\pushQED{\qed}\examplex}
  {\popQED\endexamplex}

\newcommand{\Def}[1]{\textbf{#1}}

\renewcommand{\emptyset}{\varnothing}
\renewcommand{\epsilon}{\varepsilon}
\newcommand{\R}{\mathbb{R}}
\newcommand\C{\mathbb{C}}%
\newcommand{\Rnn}{\R_{\ge 0}}
\newcommand{\Z}{\mathbb{Z}}
\newcommand{\Znn}{\Z_{\ge 0}}
\newcommand{\conv}{\mathrm{conv}}

\newcommand{\Ind}{\mathcal{I}}

\newcommand\Gr[1]{\mathrm{Gr}(#1)}%
\newcommand\Fl[1]{\mathrm{Fl}(#1)}%
\newcommand\T{T}%
\newcommand\Proj{\mathbb{P}}%
\newcommand\rspan{\mathrm{rowspan}}%
\renewcommand\t{\mathbf{t}}%

\newcommand\Baues{\mathrm{Baues}}%
\newcommand\wt{w}%
\newcommand\Nb[2]{\mathrm{Nb}_{#1}(#2)}%

\newcommand\1{\mathbf{1}}%
\newcommand\0{\mathbf{0}}%
\newcommand\rk{\mathrm{rk}}%

\newcommand\MPPl[2]{\Sigma_{#1}(#2)}%
\newcommand\MPP[1]{\Sigma_{\1}(#1)}%

\newcommand\MPv[2]{\Psi_{#1}(#2)}%

\newcommand\LF{\mathcal{L}}%
\newcommand\ov[1]{\overline{#1}}%

\newcommand\Mf[1]{\widehat{#1}}%

\newcommand\SymGrp{\mathfrak{S}}%

\newcommand\PPGKZ{\psi}%
\newcommand\PP{\Pi}%

\newcommand\inner[2]{\langle #1, #2 \rangle}%

\newcommand\kk{{\mathbf{k}}}%
\newcommand\FlagM{\mathscr{F}}%

\newcommand\Build{\mathcal{B}}%
\newcommand\wBuild{\widehat{\Build}}%

\newcommand\UC{\mathcal{U}}%

\newcommand\SYT{\mathrm{SYT}}%
\newcommand\rSYT{\mathrm{rSYT}}%

\newcommand\Forest{\mathcal{F}}%
\newcommand\dForest{\mathrm{desc}_{\Forest}}%

\DeclareMathOperator{\argmax}{argmax}

\date{\today}

\keywords{flag matroids, flag polymatroids, monotone path polytopes, generalized
permutahedra}

\subjclass[2020]{
90C57, %
05A05, %
52B12, %
90C27} %

\begin{document}

\title{Underlying Flag Polymatroids}

\author{Alexander E. Black \and Raman Sanyal}
\address[A.~Black]{Dept.\ Mathematics, Univ. of California, Davis, 
CA 95616, USA}
\email{aeblack@ucdavis.edu}

\address[R.~Sanyal]{Institut f\"ur Mathematik, Goethe-Universit\"at Frankfurt,
Germany} 
\email{sanyal@math.uni-frankfurt.de}

\begin{abstract}
    We describe a natural geometric relationship between matroids and
    underlying flag matroids by relating the geometry of the greedy algorithm
    to monotone path polytopes.  This perspective allows us to generalize the
    construction of underlying flag matroids to polymatroids. 
    We show that the polytopes associated to underlying flag polymatroid are
    simple by proving that they are normally equivalent to certain nestohedra.
    We use this to show that polymatroids realized by subspace arrangements
    give rise to smooth toric varieties in flag varieties and we interpret our
    construction in terms of toric quotients.  We give various examples that
    illustrate the rich combinatorial structure of flag polymatroids. Finally,
    we study general monotone paths on polymatroid polytopes, that relate to
    the enumeration of certain Young tableaux.
\end{abstract}

\maketitle

\section{Introduction}\label{sec:intro}
Many exciting recent
developments have benefited from the discrete geometric perspective on
matroids: For a matroid $M$ on ground set $E$ and independent sets $\Ind$, its
matroid base  polytope is
\[
    B_M \ = \ \conv\{ e_B : B \in \Ind \text{ basis} \} \ \subset \ \R^E\, .
\]
This is a $0/1$-polytope with edge directions in the type-A roots $\{ e_i -
e_j : i \neq j \}$. The geometric perspective was pioneered by Gelfand,
Goresky, MacPherson, and Serganova~\cite{GGMS}, who showed that these geometric properties
characterize matroids. Matroid base polytopes play an important role in the
interplay of combinatorics and algebraic geometry~\cite{AHK} as well as in
tropical algebraic geometry~\cite{TropGeom}; see also Section~\ref{sec:toric}.

Another important polytope associated to $M$ comes from flag matroids.
Borovik, Gelfand, Vince, and White~\cite{flagmats} introduced the
\Def{underlying flag matroid} $\FlagM_M$ of $M$ as the collection of maximal
chains of independent sets and studied them via their flag matroid polytopes
\[
    \Delta(\FlagM_M) \ = \ \conv \{ e_{I_0} + e_{I_1} + \cdots + e_{I_r} : I_0
    \subset I_1 \subset \cdots \subset I_r \in \Ind \text{ maximal chain} \}
    \, .
\]
Underlying flag matroids were also called \emph{truncation flag matroids}
in~\cite{Doker}.  Like matroid base polytopes, flag matroid polytopes are also
generalized permutahedra~\cite{GenPermOrig} that occur in connection with
torus-orbit closures in flag varieties~\cite{FlagMatAlgGeo} and they are key in
understanding tropical flag varieties~\cite{TropicalFlagVarieties,Loho}.
Underlying flag matroids are special cases of general flag matroids and
Coxeter matroids~\cite[Section~1.7]{CoxeterMat}. 

A first goal of our paper is to describe a natural geometric relationship
between these two classes of polytopes that allows us to generalize the notion
of underlying flag matroids to polymatroids.  As our notation emphasizes the
reference to the underlying matroid $M$, we will simply speak of flag matroids
henceforth.

\subsection*{Geometry of the greedy algorithm}
The well-known greedy algorithm solves linear programs over $B_M$.
Edmonds~\cite{edmondsorig} interpreted the greedy algorithm geometrically and
extended it to polymatroids. \Def{Polymatroids} are certain submodular
functions $f : 2^E \to \Rnn$ that generalize rank functions of matroids and
that naturally emerge in combinatorial optimization~\cite{Fujishige} as well
as in the study of subspace arrangements~\cite{Bjorner-Subspace}. To a
polymatroid $(E,f)$, Edmonds associated the polymatroid polytope 
\[
    P_f \ = \ \Bigl\{ x \in \R^E : x \ge 0, \sum_{e \in A} x_e \le f(A) \text{
        for all } A \subseteq E \Bigr\} \, ,
\]
If $f$ is the rank function of a matroid $M$, then $P_f = \conv\{ e_I : I \in
\Ind\}$ is the independence polytope of $M$, which we denote by $P_M$.

The base polytope $B_f$ is the face of $P_f$ that maximizes the linear
function $\1(x) = \sum_{e \in E} x_e$ and Edmonds showed that $B_f$ completely
determines $f$.  Up to translation, base polytopes $B_f$ are precisely
Postnikov's generalized permutahedra~\cite{GenPermOrig}. Edmonds' greedy
algorithm combinatorially solves the problem of maximizing $\wt \in \R^E$ over
$B_f$ by tracing a $\1$-monotone path from $0$ to a $\wt$-optimal vertex of
$B_f$; see Figure~\ref{fig:permutahedron} for an example.  
\begin{figure}[h]
    \centering
    \includegraphics[height=4cm]{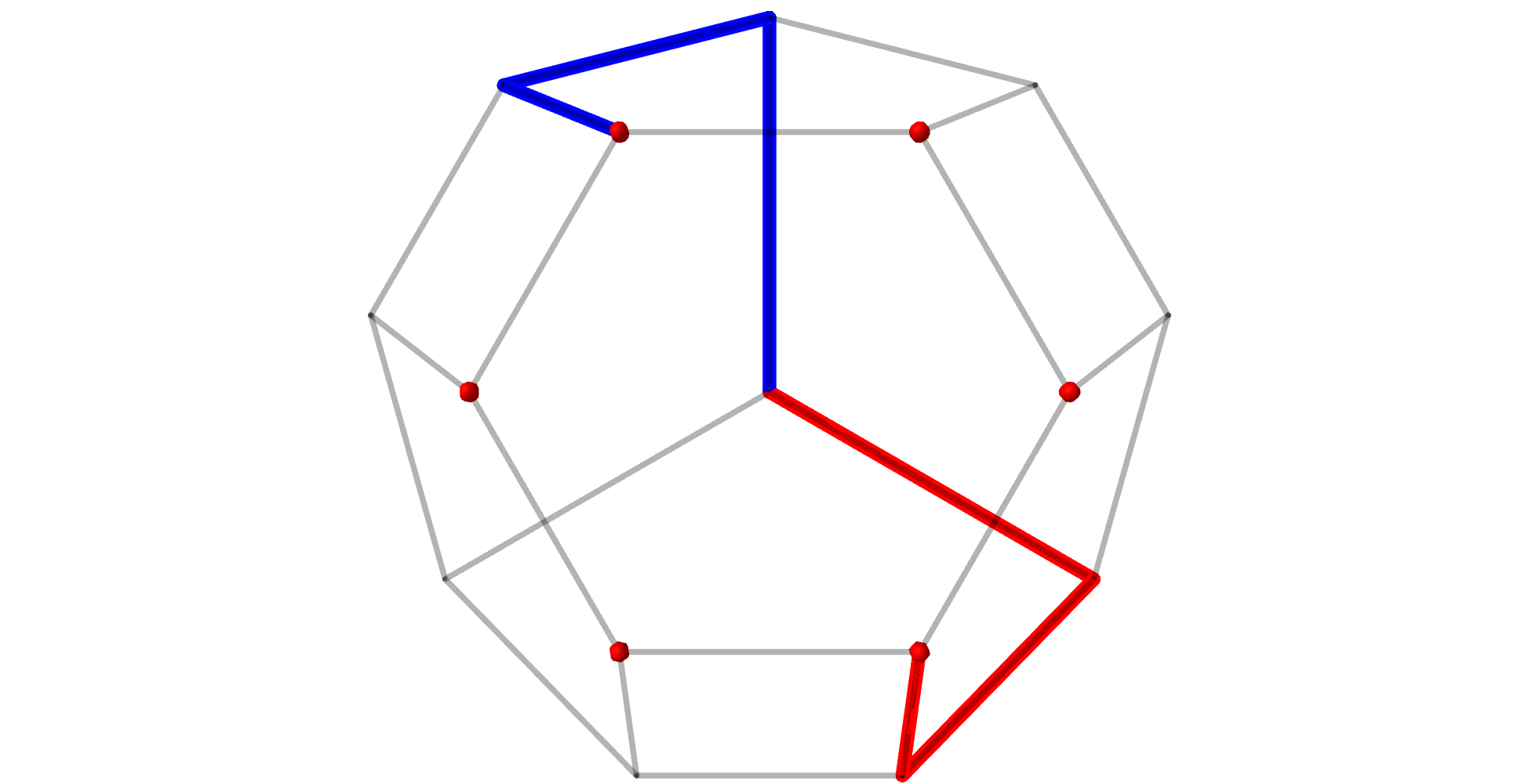}
    \caption{The gray polytope is a polymatroid polytope $P_f$. The hexagonal
    face marked by the red vertices is the base polytope $B_f$. The blue path
    and the red path are two (of six) $\1$-monotone (greedy) paths.}
\label{fig:permutahedron}
\end{figure}

So-called Baues posets capture the combinatorics of monotone
paths on polytopes and considerable attention was devoted to the topology of
Baues posets; see \Cref{sec:mp} and
\cite{CellString}. The subposet of \emph{coherent} $\1$-cellular strings on
$P_f$ is isomorphic to the face lattice of the \Def{monotone path
polytope} $\MPP{P_f}$ of Billera--Sturmfels~\cite{BSFiberPoly}. We show in
\Cref{thm:MPP_Baues} that all $\1$-cellular strings on $P_f$ are coherent and
arise from the greedy algorithm. Hence the geometry of the greedy algorithm is
completely captured by $\MPP{P_f}$. Applied to matroids, this yields the
relationship between matroid base polytopes and flag matroid polytopes.

\begin{Itheorem}\label{Ithm:relation}
    Let $M$ be a matroid. The flag matroid polytope $\Delta(\FlagM_M)$ is
    normally equivalent to the monotone path polytope $\MPP{P_M}$ of
    $\1$-cellular strings on $P_M$.
\end{Itheorem}

Normal equivalence, reviewed in Section~\ref{sec:background}, means that
$\Delta(\FlagM_M)$ and $\MPP{P_M}$ have the same underlying normal fan and, in
particular, is a strong form of combinatorial equivalence. The case of
\emph{partial} flag matroids associated to a matroid $M$ is settled by
rank-selected independence polytopes that we define in \Cref{sec:flag}.
\Cref{Ithm:relation} is then a special case of \Cref{thm:FlagDelta}.

\subsection*{Flag polymatroids}
Flag matroid polytopes are polymatroid base polytopes. Hence the behavior of
the greedy algorithm on a matroid $M$ is governed by an associated
polymatroid. 

\begin{Itheorem}[\Cref{thm:MPP_polymatroid}]\label{Ithm:MPP_polymatroid}
    Let $(E,f)$ be a polymatroid. Then $\MPP{P_f}$ is a polymatroid base
    polytope for the polymatroid 
    \[
        \Mf{f}  \ := \ 2f(E) \cdot f - f^2 \, .
    \]
\end{Itheorem}
We call $(E,\Mf{f})$ the \Def{underlying flag polymatroid} of $(E,f)$.

A flat of $(E,f)$ is a subset $F \subseteq E$ such that $f(F \cup e) > f(F)$
for all $e \in E \setminus F$. The lattice of flats $\LF(f)$ is the collection
of flats partially ordered by inclusion. For matroids, lattices of flats are
geometric lattices that completely determine the combinatorial structure of
$B_M$. For general polymatroids, this is not true. However, for flag
polymatroids it turns out that $\MPP{P_f}$ is completely determined by
$\LF(f)$. To make this more transparent, we relate flag polymatroids to yet
another class of generalized permutahedra. Postnikov~\cite{GenPermOrig} and
Feichtner--Sturmfels~\cite{nestsets} introduced nestohedra, a rich class of
simple generalized permutahedra associated to building sets $\Build \subseteq
2^E$; see \Cref{sec:pivpolys}.

\begin{Itheorem}\label{Imainthm:ppnesto}
    Let $(E,f)$ be a polymatroid. The base polytope $B_{\Mf{f}}$ of the flag
    polymatroid $(E,\Mf{f})$ is
    normally equivalent to the nestohedron for the building set
    \[
        \UC(f) \ := \ \{ E \setminus F : F \text{ flat of } f \} \, .
    \]
    In particular, flag (poly)matroid polytopes are simple polytopes.
\end{Itheorem}

In order to prove Theorem~\ref{Imainthm:ppnesto}, we make a detour via
max-slope pivot rule polytopes~\cite{PivotPoly}. We show that the greedy
algorithm on $P_f$ coincides with the simplex algorithm on $P_f$ with respect
to the max-slope pivot rule. The behavior of the max-slope pivot rule on a
fixed linear program such as $(P_f,\1)$ is encoded by an arborescence. The
arborescence represents the choices made by the pivot rule along the simplex
path started at a vertex $v$ of $P_f$ to an optimal vertex. Pivot rule
polytopes geometrically encode these arborescences. We show that $\MPP{P_f}$
is normally equivalent to the max-slope pivot rule polytope $\PP_{P_f,\1}$.
From the optimization perspective, this says the greedy path completely
determines the behavior of the max-slope pivot rule on $P_f$.
Lemma~\ref{lem:mp_implies_pr} makes that precise and might be of independent
interest. 

\subsection{Realizable polymatroids and toric quotients}

A polymatroid $(E,f)$ is realizable over $\C$ if there are linear subspaces
$(U_e)_{e \in E}$ of some common vector space such that $f(A) = \dim_\C
\sum_{e \in A} U_e$ for all $A \subseteq E$. If all subspaces are
$1$-dimensional, then $f$ is the rank function of a (realizable) matroid $M$.
Choosing an ordered bases for each $U_i$ determines a point $L_f$ in the
Grassmannian $\Gr{N,r}$ for $N = \sum_e \dim U_e$ and $r = f(E)$. We describe
the action of the algebraic torus $T^n = (\C^*)^n$ for $n = |E|$ on $\Gr{N,r}$
for which the closure of the torus orbit $T^n \cdot L_f$ is a projective toric
variety $X_f$ whose moment polytope is $B_f$ (\Cref{thm:X_f}). If $(E,f)$ is a
matroid, then this goes back to~\cite{GGMS}.

We show how a realization determines a point in the flag variety
$\Fl{N,r}$.  With respect to a suitable action of $T^n$, the torus-orbit
closure yields a projective toric variety $Y_f \subseteq \Fl{N,r}$ with moment
polytope normally equivalent to $\MPP{f}$. If $(E,f)$ is a matroid, then the
moment polytope is precisely $\Delta(\FlagM_M)$. There is a $1$-dimensional
subtorus $H \subseteq T^n$ for which $X_f / H$ is isomorphic to $Y_f$ as
topological spaces. Kapranov, Sturmfels, and Zelevinsky~\cite{KSZ} showed that
quotients of toric varieties by subtori are again toric varieties. The
associated fan is called the \emph{quotient fan} and toric varieties with this
fan are called \emph{combinatorial quotients}.

\begin{Itheorem}[\Cref{thm:comb_quot_poly}]
    The toric variety $Y_f$ is a combinatorial quotient for the action of $H$
    on $X_f$ and the moment polytope of $Y_f$ is normally equivalent to
    $\MPP{f}$. In particular, $Y_f$ is a smooth toric variety for every
    realizable polymatroid.
\end{Itheorem}

This gives an algebro-geometric explanation for the relationship between
matroid base polytopes and flag matroid polytopes.

\subsection*{Algebraic combinatorics of monotone paths on polymatroids} 

A simple variant of the greedy algorithm allows for optimization over $P_f$
and gives rise to \emph{partial} greedy paths. We show that the corresponding
monotone path polytopes are again polymatroid base polytopes. As an
application, we completely resolve a conjecture of
Heuer--Striker~\cite{partperms} on the face structure of partial permutation
polytopes (Theorem~\ref{thm:heuer_striker}).

Partial greedy paths on $P_f$ can be seen as paths on polymatroid base
polytopes.  In Section~\ref{sec:base_poly}, we investigate the combinatorics
of monotone paths on base polytopes $B_f$ with respect to the special linear
functions $\1_S(x) = \sum_{i \in S} x_i$ for $S \subseteq E$.  A case that we
study in some detail are $\1_S$-monotone paths on the permutahedron
$\Pi_{n-1}$. Let $\SYT(m,n)$ be the set of standard Young tableaux  of
rectangular shape $m \times n$. Following Mallows and
Vanderbei~\cite{VanderbeiMallows}, we call a rectangular standard Young
tableau \emph{realizable} if it can be obtained from a tropical rank-$1$
matrix; see Section~\ref{sec:base_poly} for details. Let us denote by
$\SymGrp_m$ the symmetric group on $m$ letters.

\begin{Itheorem} \label{Ithm:rSYT}
    Let $S \subseteq [n]$ with $k = |S|$.  The $\1_S$-monotone paths on the
    permutahedron $\Pi_{n-1}$ are in bijection with $\SymGrp_{k} \times
    \SymGrp_{n-k} \times \SYT(k,n-k)$. A path is coherent if and only if it
    corresponds to a realizable standard Young tableau.
\end{Itheorem}

For $k=2$, Mallows and Vanderbei showed that all $2 \times n$ rectangular
standard Young tableaux are realizable.  We give a short proof of this fact by
relating realizable standard Young tableaux to regions of the Shi arrangement
contained in the fundamental region.

\subsection*{Organization of the paper}

In Section~\ref{sec:background}, we recall notation and results on polytopes,
polymatroids, and monotone path polytopes. In Section~\ref{sec:MPP}, we show
that all cellular strings of $P_f$ are coherent and that the monotone path
polytope $\MPP{P_f}$ is a polymatroid base polytope.  We also determine the
vertices and facets. To that end, we show that $f$ and $\Mf{f}$ have the same
lattice of flats but all flats of $\Mf{f}$ are facet-defining. We illustrate
the construction on Loday's associahedron (Example~\ref{ex:associahedra}). In
Section~\ref{sec:pivpolys} we study max-slope pivot rule polytopes of
$(P_f,\1)$ and show that they are normally equivalent to $\MPP{P_f}$. We also
show that they are nestohedra for certain union-closed building sets.  In
Section~\ref{sec:flag}, we show that partial flag matroids arise as monotone
path polytopes of rank-selected independence polytopes. In
Section~\ref{sec:toric} we treat realizable polymatroids from the viewpoint of
toric varieties in Grassmannians and flag-varieties.
Section~\ref{sec:indep_greedy} extends our results to partial greedy paths and
treats a conjecture of Heuer--Striker. We close with
Section~\ref{sec:base_poly} on general monotone paths on polymatroid base
polytopes. We give plenty of examples throughout.

\subsection*{Acknowledgments}
We thank Jes\'{u}s De Loera, Johanna Krist, Georg Loho, and Milo Bechtloff
Weising for insightful conversations.  We also thank Chris Eur and Nathan
Ilten for helpful discussions regarding Section~\ref{sec:toric}.  Much of the
results obtained in this paper would not have been possible without the
\textsc{OEIS}~\cite{oeis} and \textsc{SageMath}~\cite{sage}.  The first author
is grateful for the financial support from the NSF GRFP, NSF DMS-1818969 and
the wonderful hospitality of the Goethe-Universit\"at Frankfurt and Freie
Universität Berlin, where parts of the research was conducted.

\section{Background}\label{sec:background}

In this section, we briefly recall the necessary background on polytopes,
(poly)matroids, and monotone path polytopes. For more background on polytopes,
we refer to~\cite{grunbaum} and~\cite{zieglerbook}. For a finite set $E$, the
elements of $\R^E$ are vectors $(x_a)_{a \in E}$. We will sometimes abuse
notation and identify $\R^E \cong \R^{|E|}$ with standard basis $(e_a)_{a \in
E}$ and standard inner product $\inner{x}{y} = \sum_{a \in E} x_ay_a$. For a
$A \subseteq E$, we denote by $e_A = \sum_{a \in A} e_a$ the characteristic
vector of $A$ and for any $x \in \R^E$, we write $x(A) = \1_A(x) = \sum_{a \in
A} x_a$.  We also abbreviate $\1 = \1_E$.

A polytope $P \subset \R^d$ is the convex hull of finitely many points $P =
\conv\{v_1,\dots,v_n\}$. For $\wt \in \R^d$, we write
\[
    P^\wt \ := \ \{ x \in P : \inner{\wt}{x} \ge \inner{\wt}{y} \text{ for all }
    y \in P \}
\]
for the face in direction $\wt$, that is, the set of maximizers of the linear
function $x \mapsto \inner{\wt}{x}$ over $P$. If $P^\wt = \{v\}$, then $v$ is a
vertex of $P$ and we write $V(P)$ for the set of vertices. The face lattice
$\LF(P)$ of $P$ is the collection of faces of $P$ partially ordered by
inclusion. Two polytopes $P,Q$ are combinatorially isomorphic if they have
isomorphic face lattices.

The Minkowski sum of two polytopes $P,Q \subset \R^d$ is the polytope 
\[
    P + Q \ = \ \{ p + q : p \in P, q \in Q \} \ = \ \conv\{ u + v : u \in
    V(P), v \in V(Q) \} \, .
\]
A polytope $Q$ is a \Def{Minkowski summand} of $P$ if there is a polytope $R$
such that $Q+R = P$. More generally, $Q$ is a \Def{weak} Minkowski summand if
$Q$ is a Minkowski summand of $\mu P$ for some $\mu > 0$. We will use the
following characterization of weak Minkowski summands.
\begin{prop}[{\cite[Thm.~15.1.2]{grunbaum}}]
    Let $P,Q \subset \R^d$ be polytopes. Then $Q$ is a weak Minkowski summand of
    $P$ if and only if for all $\wt \in \R^d$ it holds that $Q^\wt$ is a vertex
    whenever $P^\wt$ is.
\end{prop}

If $P$ is also a weak Minkowski summand of $Q$, then $P$ and $Q$ are
\Def{normally equivalent}. If $P$ and $Q$ are normally equivalent, then $P$ and
$Q$ are combinatorially isomorphic and the isomorphism between face lattices is
given by $P^\wt \mapsto Q^\wt$. Note that any two full-dimensional axis-parallel
boxes in $\R^d$ are normally equivalent but in general not affinely isomorphic.

\subsection{Matroids and Polymatroids}\label{sec:polymatroids}
There is a vast literature on matroids and polymatroids and we refer the
reader to~\cite{Oxley} and~\cite{Fujishige} for more.

Let $E$ be a finite set. A \Def{polymatroid}~\cite{edmondsorig} is a monotone
and submodular function $f : 2^E \to \Rnn$. That is, $f(\emptyset) = 0$ and
for all $A, B \subseteq E$
\[
    f(A) \ \le \ f(A \cup B) \ \le f(A) + f(B) - f(A \cap B) \, .
\]
The \Def{polymatroid (independence) polytope} of $f$ is
\[
    P_f \ := \ \{ x \in \R^E : x \ge 0, \1_A(x) \le f(A) \text{ for all } A
    \subseteq E \} \, .
\]
The polytope $P_f$ is of full dimension $|E|$ if and only if $f(\{e\}) > 0$
for all $e \in E$. Note that if $y \in P_f$ and $x \in \R^E$ satisfies $0 \le
x_e \le y_e$ for all $e \in E$, then $x \in P_f$. Polytopes satisfying this
condition are called \Def{anti-blocking} polytopes~\cite{Fulkerson}. 

Edmonds~\cite{edmondsorig} originally defined polymatroids as those
anti-blocking polytopes for which all points $y \in P_f$ maximal with respect
to the componentwise order have the same coordinate sum $\1(y)$. Theorem~14
in~\cite{edmondsorig} shows the equivalence to our definition above.  The
\Def{base polytope} $B_f$ of $f$ is the face $P_f^\1$. Edmonds' definition
implies that $P_f = \Rnn^E \cap (-\Rnn^E + B_f)$ and thus $B_f$ completely
determines the polymatroid. It follows from
submodularity that 
\begin{equation}\label{eqn:B_f}
     B_f \ = \ \{ x \in P_f : \1(x) = f(E) \} \, .
\end{equation}

Up to translation, base polytopes are characterized as precisely those
polytopes $B \subset \{ x : \1(x) = c \}$ for some $c$ and such that if
$[u,v]$ is an edge of $B$, then $u-v = \mu (e_i - e_j)$ for some $\mu \in \R$
and $i,j \in E$; see~\cite{edmondsorig}. In the context of geometric
combinatorics, such polytopes were studied by Postnikov~\cite{GenPermOrig}
under the name \Def{generalized permutahedra}. The prototypical examples are
permutahedra: A \Def{permutahedron} is a polytope of the form
\[
    \Pi(a_1,\dots,a_d) \ = \ \conv \{ (a_\sigma(1),\dots,a_\sigma(d)) : \sigma
    \text{ permutation of } [d] \} 
\]
for $a_1,\dots,a_d \in \R^d$; see also Example~\ref{ex:cube}.
The \Def{standard} permutahedron is $\Pi_{n-1} := \Pi(1,2,\dots,n)$.

The most well-known polymatroids are matroids. A \Def{matroid} is a pair $M =
(E, \Ind)$, where $E$ is a finite set and $\Ind \subseteq 2^E$.  The
collection $\Ind$ is a nonempty hereditary set system (or simplicial complex)
that satisfies the augmentation property: if $I,J \in \Ind$ such that $|I| <
|J|$, then there is $e \in J \setminus I$ such that $I \cup e \in \Ind$.  The
sets in $\Ind$ are called \Def{independent} and the inclusion-maximal sets are
called \Def{bases}. The \Def{rank function} of $M$ is $r_M : 2^E \to
\Z_{\ge0}$ given by $r_M(X) := \max\{ |I| : I \in \Ind, I \subseteq X \}$. The
rank function is a polymatroid with the additional property that $r_M(X) \le
|X|$ and this characterizes matroid rank functions among polymatroids.

For $A \subseteq E$, let $e_A \in \{0,1\}^E$ be its
characteristic vector. The independence polytope of a matroid $M$ is 
\[
    P_M \ := \ P_{r_M} \ = \ \conv \{ e_I : I \in \Ind \} \, .
\]
The \Def{base  polytope} of $M$ is then 
\[
    B_M \ := \ P_M^\1 \ = \ \conv \{ e_B : B \text{ basis of } M \} \, .
\]
The \Def{uniform matroid} on $n$ elements of rank $k$ is the matroid $U_{n,k}
= ([n],\Ind)$ for which a set $A \subseteq [n]$ is independent if and only if
$|A| \le k$. The corresponding base polytope is the \Def{(n,k)-hypersimplex}
\[
    \Delta(n,k) \ := \ B_{U_{n,k}} \ = \  \conv\{ e_A : A \subseteq [n], |A| =
    k \} \, .
\]

A set $F \subseteq E$ is \Def{closed} or a \Def{flat} with respect to $f$ if
$f(F \cup e) > f(F)$ for all $e \in E \setminus F$.  For $A \subset E$, the
\Def{closure} of $A$ is the flat $\ov{A} := \{ e \in E : f(A \cup e) = f(A)
\}$.  Note that $\dim P_f = |E|$ if and only if  $\emptyset$ is a flat. We
call a flat \Def{proper} if $F \neq \ov{\emptyset}$ and $F \neq E$.  The
\Def{lattice of flats} $\LF(f)$ is the collection of flats of $f$, partially
ordered by inclusion. A flat $F$ is \Def{separable} if $F = F_1 \cup F_2$ for
two disjoint, nonempty flats $F_1,F_2$ with $f(F) = f(F_1) + f(F_2)$. 

\begin{theorem}[{\cite[Thm.~28]{edmondsorig}}]
    Let $(E,f)$ be a polymatroid such that $\emptyset$ is closed. An
    irredundant inequality description of $P_f$ is given by 
    \[
        P_f \ = \ \{ x \in \R^E : x \ge 0, \1_F(x) \le f(F) \text{ for all
        proper and inseparable } F \in \LF(f) \} \, .
    \]
\end{theorem}

An operation that will be used later is the \emph{truncation} of a polymatroid:
For $0 \le \alpha \le f(E)$, the \Def{truncation} \cite[Sect.~3.1(d)]{Fujishige}
of $f$ by $\alpha$ is the polymatroid $f_\alpha$ with
\[
    f_\alpha(A) \ = \ \min(\alpha, f(A)) 
\]
The base polytope of $f_\alpha$ is $B_{f_\alpha} \ = \ P_f \cap \{ x : \1(x) =
\alpha \}$.

A matroid $M$ is \Def{realizable} over $\C$ if there are $1$-dimensional
linear subspaces $U_e \subset \C^n$ for $e \in E$ such that $r_M(X) = \dim
\sum_{e \in X} U_e$. If $U_1,\dots,U_n$ is any collection of linear subspaces,
then $f(X) = \dim \sum_{e \in X} U_e$ defines an integral polymatroid, that we
call a \Def{realizable polymatroid}. In this case $P_f$ and hence $B_f$ is a
lattice polytope.

\subsection{Monotone path polytopes}\label{sec:mp}
Let $P \subset \R^d$ be a polytope and $c \in \R^d$ a linear function that is
not constant on $P$. Let $P_{\min} =
P^{-c}$ and $P_{\max} = P^c$ be the faces on which $c$ is minimized and
maximized, respectively. A \Def{cellular string} of $(P,c)$ is a sequence of
faces $F_* = (F_0, F_1, F_2,\dots, F_r)$ of $P$ such that $c$ is not
constant on $F_i$, $F_0^{-c} \subseteq
P_{\min}, F_r^{c} \subseteq P_{\max}$, and
\[
    F_i^{c} \ =  F_i \cap F_{i+1} \ = \  F^{-c}_{i+1}
\]
for all $0 \le i < r$. If $c$ is edge generic, that is, $\inner{c}{u} \neq
\inner{c}{v}$ whenever $[u,v]$ is an edge of $P$, then the condition
simplifies to $F_0^{-c} = P_{\min}, F_r^c = P_{\max}$ and $F_i^c =
F^{-c}_{i+1}$. Cellular strings for generic $c$ were introduced and studied
in~\cite{CellString}. A partial order on cellular strings is given by
refinement, for which some $F_i$ are replaced by a cellular string of $F_i$.
For general $c$, the collection of cellular strings is still partially ordered
by refinement and we continue to call the partially ordered set the \Def{Baues
poset} $\Baues(P,c)$. The minimal elements are the \Def{$c$-monotone paths}.
They correspond to sequences of vertices $v_* = (v_0, v_1, \dots, v_k)$ such
that $v_0 \in V(P_{\min}), v_k \in V(P_{\max})$, and $[v_i,v_{i+1}] \subset P$
is an edge with $\inner{c}{v_i} < \inner{c}{v_{i+1}}$ for all $0 \le i < k$.
Figure~\ref{fig:perut_mp} gives an illustration.

Let $\wt \in \R^d$. The projection $\pi : \R^d \to \R^2$ given by $x \mapsto
(\inner{c}{x},\inner{\wt}{x})$ maps $P$ to a (degenerate) polygon $\pi(P)$.
The projections $\pi(P_{\min}), \pi(P_{\max})$ are faces of $\pi(P)$. The set
of points $(s,t) \in \pi(P)$ with $(s,t+\epsilon) \not \in \pi(P)$ for all
$\epsilon > 0$ is a vertex-edge path from the vertex $\pi(P_{\min}^\wt)$ to
the vertex $\pi(P_{\max}^\wt)$. The preimage of every edge of this path is a
cellular string, called a \Def{coherent} cellular string.  If $\wt$ is
generic, then this is a $c$-monotone path $v_* = (v_0, v_1, \dots, v_k)$,
called the \Def{shadow vertex path} of $(P,c)$ with respect to $\wt$. A
$c$-monotone path $v_*$ of $P$ is called \Def{coherent} if $v_*$ is a shadow
vertex path with respect to some $\wt$. We refer to~\cite[Section
4]{Hypersimps} for an illustration of non-coherent monotone paths

Let $I := \{ \inner{c}{x} : x \in P \} \subset \R$. A section of $(P,c)$ is a
continuous map $\gamma : I \to P$ such that $\inner{c}{\gamma(s)} = s$ for all $s \in
I$. The collection of sections is a convex body and
Billera--Sturmfels~\cite{BSFiberPoly} showed that the projection
\[
    \MPPl{c}{P} \ = \ \left\{ 2 \int_I \gamma \, ds : \gamma \text{
        section}\right\} \subset \R^d
\]
is a convex polytope, called the \Def{monotone path polytope} of $(P,c)$.  Every
$c$-monotone path $v_*$ gives rise to a piecewise-linear section
$\gamma_{v_*}$ of $(P,c)$ and 
\begin{equation}\label{eqn:MPv}
    \MPv{P,c}{v_*} \ := \ \
    2\int_I \gamma_{v_*} \, dt \ = \ 
    \sum_{j=1}^k \inner{c}{v_j - v_{j-1}} (v_{j-1} + v_j) \, .
\end{equation}
Billera--Sturmfels~\cite{BSFiberPoly} showed that a $c$-monotone path $v_*$ is
coherent with respect to $\wt$ if and only if $ \MPPl{c}{P}^\wt =
\MPv{P,c}{v_*}$.

\begin{theorem}[\cite{BSFiberPoly}] 
    The poset of coherent cellular strings is isomorphic to the face lattice of
    $\MPPl{c}{P}$.
\end{theorem}
We remark that the definition given in \cite{BSFiberPoly} is
actually $\frac{1}{2 \mathrm{vol}_1(I)} \MPPl{c}{P}$. This does not change the
combinatorics and has the benefit that if $c \in \Z^d$ and $P$ is a lattice
polytope, then $\MPPl{c}{P}$ as well.

For $s \in I$ define $P_s := \{ x \in P : \inner{c}{x} = s \}$.  The monotone path
polytope $\MPPl{c}{P}$ is equivalently given by the Minkowski integral
\begin{equation}\label{eqn:MPint}
    \MPPl{c}{P} \ = \ 2 \int_I P_s \, ds \, .
\end{equation}
Let $I' = \{ \inner{c}{v} : v \in V(P) \} = \{ t_0 < t_1 < \cdots < t_m \}$.  For $0
\le i < m$ and $t_i < s < t_{i+1}$, the polytope $P_s$ is normally equivalent
to $P_{t_i} + P_{t_{i+1}}$. The additivity of the integral gives a simple
construction for a polytope normally equivalent to the monotone path polytope.

\begin{prop}\label{prop:MPPnorm}
    The monotone path polytope $\MPPl{c}{P}$ is normally equivalent to
    $\sum_{s \in I'} P_{s}$.
\end{prop}

We give a useful local criterion of when a monotone path is coherent. Let $P$
be a polytope and $c \in \R^d$. For a vertex $v \in V(P)$, we write
\[
    \Nb{P,c}{v} \ := \ \{ u \in V(P) : [u,v] \text{ edge of $P$}, \inner{c}{u} >
    \inner{c}{v} \} 
\]
for the \Def{$c$-improving neighbors} of $v$.

\begin{lemma}\label{lem:cMP_local}
    Let $v_* = (v_0,v_1, \dots, v_k)$ be a $c$-monotone path on $(P,c)$. Then
    $v_*$ is coherent if and only if there is a weight $\wt
    \in \R^d$ such that $v_0 = (P_{\min})^\wt$ and for every $i=1,\dots,k$
    \begin{equation}\label{eqn:coMP_local}
        \frac{\inner{\wt}{v_{i}-v_{i-1}}}{\inner{c}{v_{i}-v_{i-1}}} \ > \ 
            \frac{\inner{\wt}{u-v_{i-1}}}{\inner{c}{u-v_{i-1}}} \quad \text{
                for all } u \in \Nb{P,c}{v_{i-1}} \setminus \{v_i\} \, .
    \end{equation}
\end{lemma}
\begin{proof}
    Let $\wt \in \R^d$ such that $\wt$ is not constant on $P$. The projection
    $P' = \pi(P) = \{ (\inner{c}{x}, \inner{\wt}{x}) : x \in P\}$ is a convex
    polygon and the upper hull $U$ of $P'$ is the set of points $p \in P'$
    such that $p + (0,\epsilon) \not \in P'$ for all $\epsilon > 0$. The upper
    hull is a union of edges and the corresponding coherent cellular string
    consists of the preimages of the edges of $U$ under $\pi$. If the celluar
    string is a monotone path $v_* = (v_0,v_1, \dots, v_k)$ in $P$, then
    $\pi([v_{i-1},v_i]) \subseteq U$ implies that the stated conditions
    are necessary.

    Conversely, if $u \in P$ is a vertex such that $\pi(u)$ is a vertex in
    the upper hull, then its neighbor to the right, provided it exists, is
    given by $\pi(v)$ with $\inner{c}{v} > \inner{c}{u}$ and such that $e' =
    [\pi(u),\pi(v)]$ has maximal slope. Now, in order for $\pi^{-1}(e')$ to be
    an edge, $u$ and $v$ have to be unique. This means that $v_0$ is the
    unique maximizer of $\wt$ over $P_{\min}$ and \Cref{eqn:coMP_local} has to
    be satisfied for all $i=1,\dots,k$.
\end{proof}

\section{Monotone Paths on Polymatroid Polytopes} \label{sec:MPP}

Let $(E,f)$ be a fixed polymatroid and let $\1(x) := \sum_{i \in E} x_i$. The
first goal of this section is to show that all $\1$-cellular strings on $P_f$
are coherent. We write $\MPP{f} := \MPP{P_f}$ for the monotone path polytope of
$P_f$ with respect to $\1$.

\begin{theorem}\label{thm:MPP_Baues}
    Let $(E,f)$ be a polymatroid.  Every $\1$-cellular string on $P_f$ is
    coherent. In particular, the Baues poset $\Baues(P_f,\1)$ is isomorphic to
    the face lattice of $\MPP{P_f}$.
\end{theorem}

We will identify $E = \{1,\dots,n\}$. Edmonds~\cite{edmondsorig_old} showed
that the following geometric version of the greedy algorithm can be used on
polymatroids\footnote{The paper was again published in the Edmonds
Festschrift~\cite{edmondsorig} and throughout we will reference the results
there.}.

\begin{theorem}[Greedy Algorithm]\label{thm:greedy}
    Let $(E,f)$ be a polymatroid and $\wt \in \R^E$. Let $\sigma$ be a
    permutation such that $\wt_{\sigma(1)} \ge \wt_{\sigma(2)} \ge \cdots \ge
    \wt_{\sigma(n)}$. For $i = 0,\dots,n$ define $A_i := \{ \sigma(1), \dots,
    \sigma(i) \}$ and $x \in \R^E$ by 
    \[
        x_{\sigma(i)} \ := \ f(A_{i}) - f(A_{i-1})
    \]
    for $i=1,\dots,n$. Then $x$ maximizes $\wt$ over the base polytope $B_f$.
    If the greedy algorithm is stopped at $\wt_{\sigma(i)} < 0$ and
    $x_{\sigma(j)} := 0$ for $j \ge i$, then $x$ maximizes $\wt$ over $P_f$.
\end{theorem}

In particular every vertex of $P_f$ and $B_f$ can be found using the greedy
algorithm. For a vertex $v \in V(P_f)$, the support $I(v) := \{ i \in E : v_i
> 0 \}$ is called the \Def{basis} of $v$. This is rarely a closed set. For
example, if $v \in B_f$, then $\ov{I(v)} = E \setminus \ov{\emptyset}$.

Let $v$ be the vertex of $B_f$ obtained from the greedy algorithm with respect
to $\wt$. We can assume that $\wt$ is generic.  Let $I(v) = \{ j_1, j_2,
\dots, j_k \}$ so that $w_{j_1} > w_{j_2} > \cdots > w_{j_k}$. Define $F_0
\subset F_1 \subset \cdots \subset F_k$ by $F_i := \ov{\{j_1,\dots,j_i\}}$ for
$i=0,\dots,k$ and define $0 = v_0,v_1,\dots,v_k= v$ by setting
\[
    (v_{i})_j \ = \ 
    \begin{cases}
        f(F_{j_h}) - f(F_{j_{h-1}}) & \text{if } j = j_h, h \le i \\
        0 & \text{otherwise.}
    \end{cases}
\]
Theorem~\ref{thm:greedy} implies that $v_0,\dots,v_k$ are distinct vertices of
$P_f$ such that $\ov{I(v_i)} = F_i$, $\sum_j (v_i)_j = f(F_i)$, and
$[v_i,v_{i+1}]$ is a $\1$-increasing edge of $P_f$.  Note that the
$\1$-monotone path is completely determined by the ordered sequence $j_* =
(j_1,j_2,\dots,j_k)$. We call $v_0,\dots,v_k$ or, equivalently, $j_*$ a
\Def{greedy path} of $P_f$.

\begin{prop} \label{thm:greedycoh}
    Let $(E,f)$ be a polymatroid and $\wt \in \R^E$ generic. The greedy path
    associated to $\wt$ is a coherent $\1$-monotone path.
\end{prop}
\begin{proof}
    Let $v' \in V(P_f)$ be a vertex. If $u$ is a neighbor of $v$ with $\1(u) >
    \1(v)$, then $u-v' = \delta e_i$ for some $i=1,\dots,n$ and $\delta > 0$.
    Hence,
    \[
        \frac{\inner{\wt}{u-v'}}{\1(u-v')} \ = \ \frac{\delta \wt_i}{\delta} \
        = \ \wt_i
    \]
    and~\eqref{eqn:coMP_local} implies that the coherent monotone path $0 =
    v_0,v_1,\dots,v_k = v$ of $P_f$ is precisely the path obtained from the
    greedy algorithm.
\end{proof}

\begin{prop}
    Every $\1$-monotone path on $P_f$ is a greedy path.
\end{prop}
\begin{proof}
    Let $\0 = v_0, v_1, \dots, v_s$ be a $\1$-monotone path on $P_f$.
    Then $v_i - v_{i-1} = \delta_i e_{j_i}$ for $i = 1,\dots,s$. Choose a
    weight $\wt$ with $\wt_{j_1} > \wt_{j_2} > \cdots > \wt_{j_s} > \wt_h$ for
    $h \not \in \{j_1,\dots,j_s\}$. Since $v_i$ is a vertex of the truncation
    $P_{f_\alpha}$ for $\alpha = \1(v_i)$, the greedy path with respect to
    $\wt$ will be precisely the given monotone path.
\end{proof}

\begin{proof}[Proof of Theorem~\ref{thm:MPP_Baues}]
    Let $\0 = F_0, F_1, F_2,\dots, F_k$ be a $\1$-cellular string on $P_f$.
    For $h = 1,\dots,k$ define 
    \[
        I_h \ := \ \{ i \in E : p + \delta e_i \in F_i \text{ for some } p \in
        F_{h-1} \cap F_h \text{ and } \delta > 0 \} \, .
    \]
    We claim that the cellular string is completely determined by
    $I_1,\dots,I_k$. Indeed, let $L_h = \mathrm{span}\{ e_i : i \in I_h\}$.
    Then $F_1 = P_f \cap L_1$. If, by induction,  $F_h$ is determined, then we
    can employ the greedy algorithm to find a point $p$ in $F_h^\1 = F_h \cap
    F_{h+1}$ and $F_{h+1} = P_f \cap (p + L_{h+1})$. Again the greedy
    algorithm shows that $F_0, F_1,\dots,F_k$ is precisely the coherent
    cellular string for $\wt = k e_{I_1} + (k-1) e_{I_2} +  \cdots + e_{I_k} -
    e_{E \setminus I_k}$.
\end{proof}

\begin{remark}
    Note that the linear function $c = \1$ is essential for the validity of
    Theorem~\ref{thm:MPP_Baues}. Consider, for example the uniform matroid
    $U_{4,2}$ with rank function $f(A) = \min(|A|,2)$ for $A \subseteq [4]$.
    The polymatroid polytope $P_f$ is the convex hull of all $v \in \{0,1\}^4$
    with at most two entries equal to $1$.  The linear function $c =
    (-10,-5,7,8)$ is generic on the polymatroid (independence) polytope $P_f$
    with minimum $v_{\min} = (1,1,0,0)$ and maximum $v_{\max} = (0,0,1,1)$. It
    can be checked that, for example using Lemma~\ref{lem:cMP_local}, that the
    $c$-monotone path $(1, 1, 0, 0), (1, 0, 1, 0), (1, 0, 0, 1), (0, 1, 0, 1),
    (0, 0, 1, 1)$ is not coherent.
    Note that the monotone path is contained in the base polytope $B_f$.
    In Section~\ref{sec:base_poly} we focus on monotone paths in base
    polytopes.
\end{remark}

We give a complete combinatorial description of $\MPP{f}$ in
Section~\ref{sec:pivpolys}. Here, we only describe the vertices and
facet-defining inequalities.  The greedy algorithm readily gives a purely
combinatorial description of the vertices of $\MPP{f}$. 

\begin{cor}\label{cor:comb_vertices}
    Let $(E,f)$ be a polymatroid. The vertices of $\MPP{f}$ are in
    correspondence with sequences $j_* = (j_1,j_2,\dots,j_k)$ of distinct
    elements of $E$ such that 
    \[
        \ov{\emptyset} \ \subset \ 
        \ov{\{j_1\}} \ \subset \ 
        \ov{\{j_1,j_2\}} \ \subset \ 
        \cdots \ \subset \ 
        \ov{\{j_1,j_2,\dots,j_k\}} \ = \ E
    \]
    is a maximal chain of flats in $\LF(f)$.
\end{cor}

Corollary~\ref{cor:comb_vertices} also prompts an organizing principle to
group vertices which produce the same maximal chain of flats. For a sequence
$j_* = (j_1,j_2,\dots,j_k)$ define $F_i(j_*) :=\ov{\{j_1,j_2,\dots,j_i\}}$ for
$i = 0,1,\dots,k$. 
Using~\eqref{eqn:MPv}, a direct computation yields the vertices of $\MPP{f}$.

\begin{cor}\label{cor:comb_vertices_coord}
    Let $j_* = (j_1,\dots,j_k)$ be a $\1$-monotone path of $P_f$ and let $F_i
    = F_i(j_*)$ for $i=0,\dots,k$. The vertex $\Psi(j_*)$ of $\MPP{f}$
    corresponding to the greedy path $j_*$ satisfies $\Psi(j_*)_r = 0$ if $r
    \not\in \{j_1,\dots,j_k\}$ and 
    \[
        \Psi(j_*)_{j_i} \ = \ (f(F_i) - f(F_{i-1})) (f(E) - f(F_i) + f(E) -
        f(F_{i-1})) \, .
    \]
\end{cor}

Our next goal is to show that monotone path polytopes of polymatroid polytopes
are polymatroid base polytopes. Figure~\ref{fig:perut_mp} gives a first
illustration.
\begin{figure}[h]
    \centering
    \includegraphics[width=7cm]{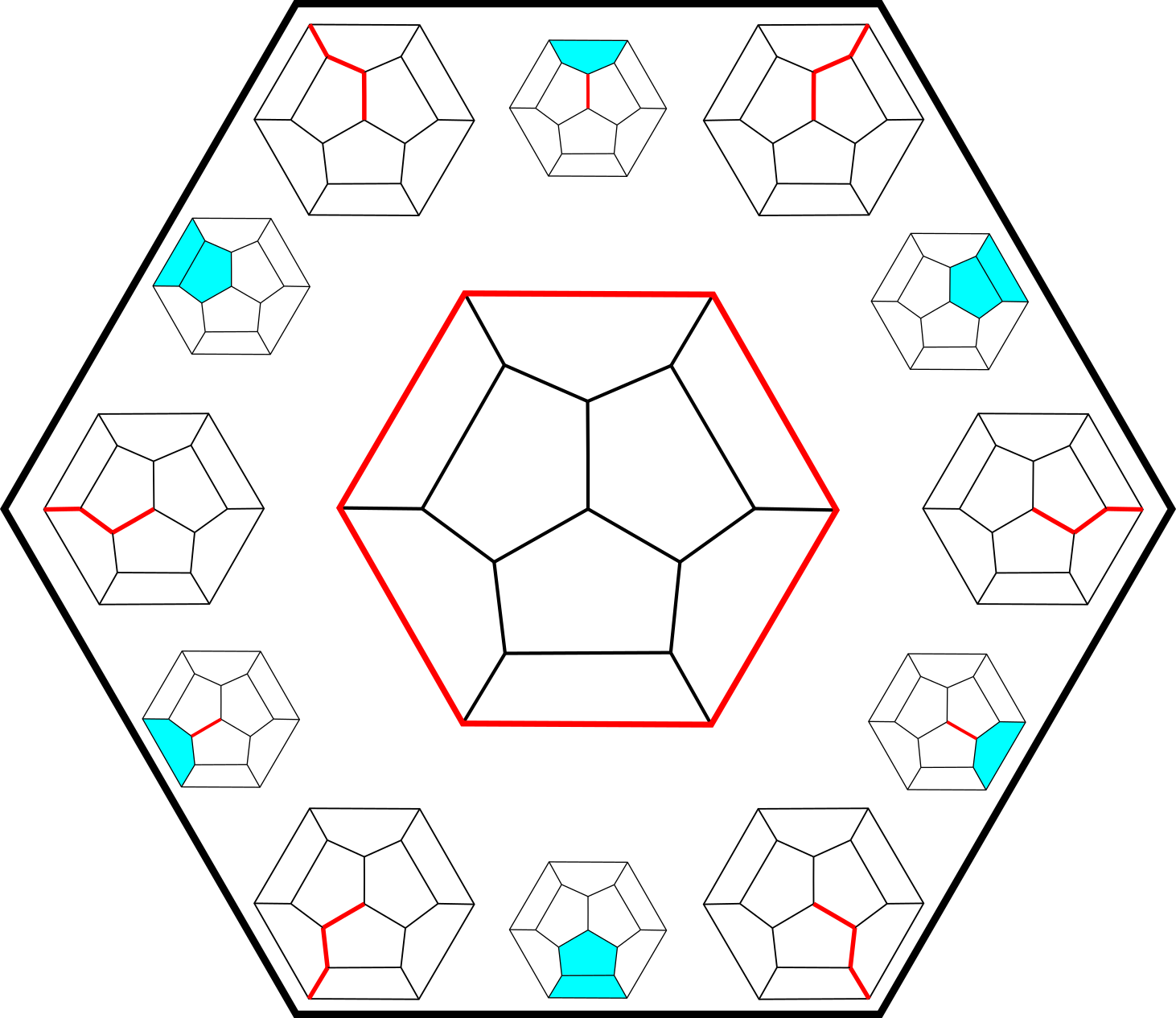}
    \caption{The black hexagon is the monotone path polytope $\MPP{f}$ of the
    polymatroid polytope in Figure~\ref{fig:permutahedron}. The correspondence
    between cellular strings and faces of $\MPP{f}$ is indicated.}
\label{fig:perut_mp}
\end{figure}

\begin{theorem}\label{thm:MPP_polymatroid}
    Let $(E,f)$ be a polymatroid. Then $\MPP{f}$ is a polymatroid base
    polytope for the polymatroid 
    \[
        \Mf{f}  \ := \ 2f(E) \cdot f - f^2 \, .
    \]
\end{theorem}

Theorem~\ref{thm:MPP_polymatroid} is a first justification of calling $\Mf{f}$
a \Def{flag polymatroid} associated to $f$. Note that the transformation $f
\mapsto \Mf{f}$ is homogeneous of degree $1$, that is, $\Mf{ \alpha f } =
\alpha\Mf{f}$ for all $\alpha > 0$.

\begin{proof}
    By the characterization~\eqref{eqn:MPint}
    \[
        \MPP{f} 
        \ = \ 2\int^{f(E)}_0 P_f \cap \{ x : \1(x) = t \}\, dt 
        \ = \ \int^{f(E)}_0 2B_{f_t} \, dt  \, ,
    \]
    where $B_{f_t}$ is the base polytope of the truncation $f_t(A) = \min(f(A),
    t)$; see
    Section~\ref{sec:polymatroids}.  Since polymatroid base polytopes are
    closed under Minkowski sums, it follows that $\MPP{P_f} = B_g$ for
    some submodular function $g$. In order to determine $g$, we compute for $S
    \subseteq E$
    \[
        g(S) \ = \ \int_{0}^{f(E)} 2f_t(S) \, dt
        \ = \ 
        \int_{0}^{f(S)} 2t \, dt + 
        \int_{f(S)}^{f(E)} 2f(S) \, dt
        \ = \ f(S)^2 + 2f(S)f(E) - 2f(S)^2 \, ,
    \]
    which finishes the proof.
\end{proof}

Via~\eqref{eqn:B_f}, Theorem~\ref{thm:MPP_polymatroid} gives an inequality
description. We next determine the inseparable flats.

\begin{prop}\label{prop:same_flats}
    Let $(E,f)$ be a polymatroid. For $A,B \subseteq E$ we have 
    \[
        \Mf{f}(A) \ = \  \Mf{f}(B) 
        \quad \text{ if and only if } \quad
        f(A) \ = \ f(B)  \, .
    \]
    In particular, $f$ and $\Mf{f}$ have the same lattices of flats.
\end{prop}
	
\begin{proof}
    Consider the function $g(t) := 2t - t^2$, which is an injective
    function on $[0,1]$. We may assume that $f(E) = 1$ so that $\Mf{f} =
    g(f)$ and the result follows.
\end{proof}

\begin{prop}\label{prop:all_insep}
    Let $(E,f)$ be a polymatroid and $\Mf{f}$ its flag polymatroid.
    Every flat $A$ of $\Mf{f}$ is inseparable.
\end{prop}
\begin{proof}
    We may assume that $f(E) = 1$. Let $A$ be a fixed flat with $a = f(A) \le 1$ and 
    assume that $A$ is separable with respect to $\Mf{f}$. That is, there are
    disjoint flats $A_1,A_2 \subseteq A$ such that $\Mf{f}(A) = \Mf{f}(A_1) +
    \Mf{f}(A_2)$.
    Then $(a_1,a_2) = (f(A_1),f(A_2))$ satisfies
    \[
        2a_1 - a_1^2 + 2a_2 - a_2^2 \ = \ 2a - a^2 
        \quad \Longleftrightarrow \quad
        (1-a_1)^2 + (1-a_2)^2 \ = \ (1-a)^2 + 1 \, .
    \]
    Monotonicity and submodularity yield $0 \le a_1,a_2 \le a$ and $a \le a_1
    + a_2$.  Reparametrizing $(a_1,a_2,a) = (1-b_1,1-b_2,1-b)$, we are thus
    looking at pairs $(b_1,b_2)$ such that 
    \[
        b \le b_1,b_2 \le 1 
        \quad \text{ and } \quad 
         b_1 + b_2  \ \le \ 1+b
        \quad \text{ and } \quad 
        b_1^2 + b_2^2 \ = \ 1+ b^2  \, .
    \]
    The linear inequalities describe a triangle in the plane contained in the
    disc with radius $\sqrt{1+b^2}$ and meeting the bounding circle in the
    points $(1,b)$ and $(b,1)$. This, however, implies that $\Mf{f}(A_1) =
    \Mf{f}(A)$ or $\Mf{f}(A_2) = \Mf{f}(A)$ and hence $A = A_1$ or $A = A_2$.
    This shows that $A$ is inseparable.
\end{proof}

Theorem~\ref{thm:MPP_polymatroid} together with the last two propositions give
an irredundant inequality description:
\[
    \MPP{f} \ = \ \Bigl\{ x \in \R^E : x \ge 0, \1(x) = f(E)^2, \1_F(x) \le
    2f(E)f(F) - f(F)^2 \text{ for all proper } F \in \LF(f) \Bigr\}
    \, .
\]

\begin{cor}
    Let $F \in \LF(f)$ be a flat. The vertices of the facet $\MPP{f}^{e_F}$
    are precisely the greedy paths that pass through the flat $F$.

    Moreover, let $\emptyset = F_0 \subset \cdots \subset F_k = E$ a maximal
    chain of flats in $\LF(f)$. The collection of vertices $j_*$ with
    $F_i(j_*) = F_i$ for $i=0,\dots,k$ form a face of $\MPP{f}$
    combinatorially isomorphic to a product of simplices of dimensions $|F_i
    \setminus F_{i-1}|-1$ for $i = 1,\dots,k$.
\end{cor}

\begin{example}[Matroids]
    Let $M$ be a rank-$r$ matroid on ground set $E$.  It follows from
    Corollary~\ref{cor:comb_vertices} that a sequence \mbox{$j_* =
    (j_1,\dots,j_k)$} is a greedy path if and only if $k = r$ and
    $\{j_1,\dots,j_i\}$ is independent in $M$ for $i=0,\dots,r$. In
    particular, $(j_1,\dots,j_r)$ is an ordered basis of $M$.  Using
    Corollary~\ref{cor:comb_vertices_coord} together with the fact that
    $r_M(\ov{\{j_1,\dots,j_i\}}) = i$, we find that the vertex of
    $\MPP{P_M}$ corresponding to the greedy path is 
    \[
        (2r-1) e_{j_1} + 
        (2r-3) e_{j_2} + 
        \cdots +  3 e_{j_{r-1}} 
        +  e_{j_r} 
    \]
    If $B \subseteq E$ is a basis of $M$, then the face $\MPP{M}^{e_B}$ is
    linearly isomorphic to $-\1 + 2\Pi_{r-1}$. We will come back to this
    example in the next sections.
\end{example}

\begin{example}[Cubes and permutahedra]\label{ex:cube}
    For $E = [n]$, let $f : 2^E \to \Znn$ be the polymatroid given by $f(A) =
    |A|$. This is the rank function of the uniform matroid $U_{n,n}$ and $P_f
    = [0,1]^n$ with $B_f = \{ (1,\dots,1)\}$. The greedy paths are given by
    all permutations $(\sigma(1),\dots,\sigma(n))$. The flag polymatroid is
    $\Mf{f}(A) \ = \ n^2 - (n - |A|)^2$ and using
    Corollary~\ref{cor:comb_vertices_coord}, we see that the vertices of
    $\MPP{f}$ are the permutations of $(1,3,\dots,2n-1)$.  Hence $\MPP{f} =
    -\1 + 2 \Pi_{n-1}$.
\end{example}

Notice that if $(E,f)$ is a polymatroid with $\LF(f) = 2^E$, then $\MPP{f}$
has $2^{|E|}-2$ facet-defining inequalities and hence is normally equivalent
to the permutahedron. 

\begin{prop}
    If $f(E) - f(E \setminus i) > 0$ for all $i \in E$, then $\MPP{f}$ is
    normally equivalent to the permutahedron.
\end{prop}
\begin{proof}
    Assume that $B_f \subset \R^E_{>0}$.  Any two subsets $A \subset A'$ with
    $|A'| = |A| + 1$ occur in some execution of the greedy algorithm
    (Theorem~\ref{thm:greedy}) and lead to a vertex $v \in B_f$. It thus
    follows that $f(A) < f(A')$ and  $\LF(f) = 2^E$. Now $B_f \subset
    \R^E_{>0}$ if and only if the maximum of the linear function $x \mapsto
    -x_i$ is positive over $B_f$ for all $i$. The greedy algorithm shows that
    this is the case if and only if $f(E) - f(E \setminus i) > 0$ for all $i
    \in E$.
\end{proof}

We call a polymatroid $f$ \Def{tight} if $f(E \setminus i) = f(E)$ for all $i
\in E$.

\newcommand\Ass{\mathrm{Ass}}%
\newcommand\oAss{\overline{\Ass}}%
\begin{example}[Associahedra]\label{ex:associahedra}
    Let $n \ge 1$. For $1 \le i \le j \le n$, we write $\Delta_{[i,j]} =
    \conv\{e_i,e_{i+1},\dots,e_j\}$.  The \Def{Loday
    Associahedron}~\cite{Loday} is the polymatroid base polytope
    \[
        \Ass_{n-1} \ := \ \sum_{1 \le i \le j \le n} \Delta_{[i,j]} \, .
    \]
    More precisely, $\Ass_{n-1}$ is a nestohedron; see next section
    and~\cite[Sect.~8.2]{GenPermOrig}.  The underlying polymatroid
    $([n],f_\Ass)$ is given by
    \[
        f_\Ass(A) \ := \ |\{ 1 \le i \le j \le n : \{i,\dots,j\} \cap A \neq
        \emptyset \}|
    \]
    for $A \subseteq [n]$. The vertices of $\Ass_{n-1}$ are in bijection with
    plane binary trees. For a generic weight $\wt \in \R^n$, let $i \in [n]$
    with $\wt_i$ maximal. The vertex $v$ of $\Ass_{n-1}$ maximizing $\wt$
    corresponds to the plane binary $T$ with root $i$ and left and right
    subtree recursively determined by $(w_1,\dots,w_{i-1})$ and
    $(w_{i+1},\dots,w_n)$, respectively. Let $\sigma$ be the unique
    permutation such that $w_{\sigma^{-1}(n)} > w_{\sigma^{-1}(n-1)} > \cdots
    > w_{\sigma^{-1}(1)}$.  Then, viewed as a linear function $\sigma \in
    \R^n$, $\sigma$ determines the same binary tree. The permutation
    determines how $T$ is built up. Hence every permutation represents a
    different greedy path and hence $\MPP{f_\Ass}$ is normally equivalent to a
    permutahedron.

    To see this differently, let $T$ be a plane binary tree and let $L_j$
    and $R_j$ be the number of nodes in the left, respectively, right subtree
    of $T$ rooted at $j$. The vertex $v$ of $\Ass_{n-1}$ corresponding to $T$
    has coordinates $v_j = (L_j + 1)(R_j + 1)$  \cite[Cor.~8.2]{GenPermOrig}.
    In particular, $f_\Ass$ is not tight.

    The number of greedy paths that lead to a fixed tree $T$ can be computed
    as follows. View $T$ as a poset where the minimal elements are precisely
    the leaves of $T$. A linear extension is a permutation $\sigma$ with
    $\sigma(i) < \sigma(j)$ whenever $j$ is on the path from $i$ to the root.
    The greedy paths leading to $T$ are precisely the linear extensions of $T$.
    The number of linear extensions can be computed by the \emph{tree
    hook-length formula} \cite[Exercise
    5.1.4.(20)]{Knuth}\cite[Prop.~22.1]{Stanley-Ordered}
    \[
        e(T) \ = \ n! \prod_{i=1}^n \frac{1}{(L_i+R_i+1)} \, .
    \]

    Consider the polytope
    \[
        \oAss_{n-1} \ := \ \sum_{1 \le i < j \le n} \Delta_{[i,j]} \ = \ -\1 +
        \Ass_{n-1} \, .
    \]
    This is a tight version of the associahedron with $f_{\oAss}(A) =
    f_\Ass(A) - |A|$. The polytopes $\Ass_{n-1}$ and $\oAss_{n-1}$ differ only
    by a translation but their polymatroid polytopes and their flag 
    polymatroids are different. For a binary tree $T$ let $T'$ be the
    tree obtained from $T$ by removing all leaves.  Two permutations $\sigma^1$
    and $\sigma^2$ yield the same greedy path on $P_{f_{\oAss}}$ if and only
    if both are linear extensions of $T$ and they yield the same linear
    extension of $T'$ after relabelling. The number of vertices of
    $\MPP{f_{\oAss}}$ is then
    \[
        \sum_{T} e(T')  \, ,
    \]
    where the sum is over all plane binary trees on $n$ nodes. The first few
    numbers starting with $n = 2$ are $2, 5, 14, 46, 176, 766, 3704, 19600,
    112496$.
\end{example}

\newcommand\PM{\mathscr{P}^1}%
Let us close this section with the observation that the flag polymatroid
defines a nonlinear transformation on the space of polymatroids.  For
instance, let $\PM_n$ be the compact convex set of polymatroids $f : 2^{[n]}
\to \Rnn$ with $f([n]) = 1$. Then $f \mapsto \Mf{f} = 2f - f^2$ defines a
discrete dynamical system on $\PM_n$.

\begin{prop}
    Let $f \in \PM_n$ be a polymatroid for which $\emptyset$ is closed. Define
    $f^0 := f$ and $f^{i+1} := 2\Mf{f^i}$. The sequence $(f^i)_{i \ge 0}$
    converges to the function $f^\infty \in \PM_n$ with $f^\infty(A) = 1$ for
    all $A$.
\end{prop}
\begin{proof}
    It follows from Theorem~\ref{thm:MPP_polymatroid} that $f^{i+1} \in
    \PM_n$. Since $g(t) = 2t - t^2$ is strictly increasing on the interval
    $(0,1)$, we have $f^{n}(A) \le f^{n+1}(A) = g(f^n(A))$ for all $A
    \subseteq [n]$ and with strict inequality unless $f(A) = 1$. Now if
    $\emptyset$ is closed, this implies $f(A) > 0$ for all $A \neq \emptyset$.
\end{proof}

\newcommand\arb{\mathcal{A}}%
\section{Pivot Polytopes and Nestohedra} \label{sec:pivpolys}

In the context of linear optimization, the authors, De Loera, and
L\"{u}tjeharms introduced \emph{pivot polytopes} in~\cite{PivotPoly}. Let $P
\subset \R^n$ be a fixed polytope with vertex set $V(P)$. Recall that
$\Nb{P,c}{v}$ is the collection of $c$-improving neighbors of $v \in V(P)$.
For a fixed linear function $c$, a memory-less pivot rule for the pair $(P,c)$
is a map $\arb : V(P) \to V(P)$ such that $\arb(v) = v$ for all vertices $v$
maximizing $c$ and $\arb(v) \in \Nb{P,c}{v}$ otherwise. If $c$ is generic,
then $\arb$ is an arborescence of the graph of $P$ with acyclic orientation
induced by $c$. For the simplex algorithm, $\arb(v)$ encodes the choices made
by a (memory-less) pivot rule at the vertex $v$. We refer the reader
to~\cite{PivotPoly} for details.  We abuse notation and will refer to such
maps $\arb$ as \Def{arborescences} for the pair $(P,c)$ even when the
maximizer of $c$ over $P$ is not unique.

For a weight $\wt \in \R^n$ linearly independent of $c$, the \Def{max-slope
pivot rule} on $(P,c)$ corresponds to the arborescence $\arb^\wt_{P,c}$
determined by 
\begin{equation}\label{eqn:MSarb}
    \arb^\wt_{P,c}(v) \ = \ \argmax \Bigl\{
        \frac{\inner{\wt}{u-v}}{\inner{c}{u-v}} : u \in \Nb{P,c}{v} \Bigr\} \, .
\end{equation}

For an arborescence $\arb$, define 
\begin{equation}\label{eqn:PPGKZ}
    \PPGKZ(\arb) \ := \ \sum_{v} \frac{\arb(v) - v}{\inner{c}{\arb(v) -v}} \,,
\end{equation}
where we tacitly assume that $\frac{0}{\inner{c}{0}} = 0$. The \Def{max-slope pivot
rule polytope} is the polytope
\begin{equation}\label{eqn:PP}
    \PP_{P,c} \ := \ \conv \{ \PPGKZ(\arb) : \arb \text{
        arborescence of $(P,c)$} \} \, .
\end{equation}

\begin{theorem}[{\cite[Theorem~1.4]{PivotPoly}}]
    The vertices of $\PP_{P,c}$ are in one-to-one correspondence to the
    max-slope arborescences of $(P,c)$.
\end{theorem}

We can canonically decompose $\PP_{P,c}$ into a Minkowski sum
\begin{equation}\label{eqn:PPsum}
    \PP_{P,c} \ = \ \sum_{v \in V(P)} \PP_{P,c}(v) \, ,
\end{equation}
where
\begin{equation}\label{eqn:PPv}
    \PP_{P,c}(v) \ := \ \conv\Bigl\{ \frac{u-v}{\inner{c}{u-v}} : 
    u \in \Nb{P,c}{v} \Bigr\} \, .
\end{equation}

The max-slope pivot rule polytope is intimately related to the monotone path
polytope $\MPPl{c}{P}$. For a generic $\wt$, let $v_0 = (P^{-c})^\wt$
and define $v_i := \arb^\wt_{P,c}(v_{i-1})$ for $i \ge 1$. If $k$ is
minimal with $v_k = v_{k+1}$, then $v_0,v_1 \dots, v_k$ is the
coherent monotone path of $(P,c)$ with respect to $\wt$. From this, we
deduced the following geometric implication.

\begin{prop}[{\cite[Theorem~1.6]{PivotPoly}}] \label{prop:weakMink} 
    The monotone path polytope $\MPPl{c}{P}$ is a weak Minkowski summand of the
    max-slope pivot rule polytope $\PP_{P,c}$.
\end{prop}

We now show that the converse relation also holds for $(P_f,\1)$. 

\begin{theorem}\label{mainthm:ppnormmp}
    Let $(E,f)$ be a polymatroid. Then $\MPP{P_f}$ is normally equivalent to
    $\PP_{P_f,\1}$.
\end{theorem}

From the perspective of optimization, Theorem~\ref{mainthm:ppnormmp} implies
the following.

\begin{cor}\label{cor:max_slope_greedy}
    The greedy algorithm on $P_f$ corresponds to linear optimization on
    $(P_f,\1)$ with respect to the max-slope pivot rule.
\end{cor}

We start by making an observation about the behavior of the greedy algorithm.
Any generic $\wt \in \R^E$ induces a total order $\preceq$ on $E$ by setting
$i \prec j$ if $\wt_i > \wt_j$. The greedy algorithm on $P_f$ with respect to
$\wt$ produces a vertex $u \in V(B_f)$. We call $I(u) = \{b_1 \prec b_2
\prec \cdots \prec b_k\}$ the \Def{optimal basis} of $f$ with respect to
$\wt$.
	
\begin{lemma} \label{lem:mp_implies_pr} 
    Let $(E,f)$ be a polymatroid with total order $\preceq$ and optimal basis
    $B$. Let $v$ be a vertex of $P_f$ and $j \in E$ $\prec$-minimal with the
    property that $v + \lambda e_j \in P_f$ for some $\lambda > 0$. Then $j
    \in B$.
\end{lemma}

Geometrically, the lemma states that if we start the geometric greedy
algorithm at a vertex $v$, then the set of directions taken is a subset of the
directions taken from the vertex $0$ along the greedy path.
Figure~\ref{fig:permut_pivot} shows this for the polymatroid polytope of the
permutahedron $\Pi_2$.
\begin{figure}[h]
    \centering
    \includegraphics[width=.9\textwidth]{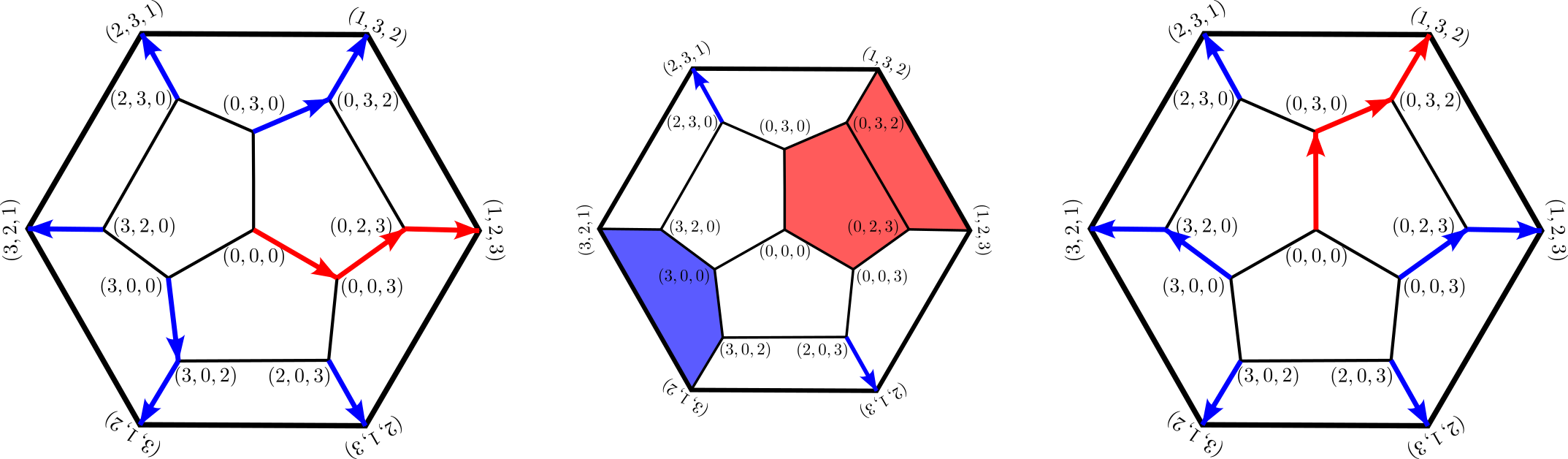}
    \caption{The figures left and right show two max-slope arborescences of the
    polymatroid polytope of the $2$-dimensional permutahedron. The red paths
    are the greedy paths. The arborescences are adjacent on the pivot rule
    polytope. The middle figure shows the multi-arborescence corresponding to
    the edge. The cellular string is shown in red.}
\label{fig:permut_pivot}
\end{figure}
	
\begin{proof}
    Assume that $j \not\in B$. Let $B^+ := \{b \in B: b \prec j\}$. Since $j$
    is not added to $B$ by the greedy algorithm, we have $B^+ \neq \emptyset$
    and $f(B^+ \cup j) = f(B^+)$. Let $I = I(v) = \{ i \in E : v_i >0 \}$ be
    the basis of $v$. Since $j$ is the next direction chosen at $v$, $f(I \cup
    b) = f(I)$ for each $b \in B^+$. Monotonicity and submodularity implies
    $f(I \cup B^+) = f(I)$. Again by monotonicity and submodularity, 
    \[
        f(I) \le f(I \cup j) \le \
        f(I \cup B^+ \cup j) \le f(I \cup B^+) + f(B^+ \cup j)
        - f(B^+) = f(I \cup B^+) = f(I) \, 
    \]
    which contradicts the fact that  $v + \lambda e_j \in P_f$ for $\lambda
    >0$.
\end{proof}

\begin{proof}[Proof of Theorem~\ref{mainthm:ppnormmp}]
    We may assume that $P_f$ is full-dimensional.
    We need to show for every weight $\wt$ that $(\PP_{P_f,\1})^\wt$ is a vertex
    whenever $\MPP{f}^w$ is a vertex. To that end, let $\MPP{f}^w$ be a vertex
    corresponding to a coherent monotone path of $P_f$ with respect to $\wt$.
    The path is encoded by the optimal basis $B = (j_1 \prec j_2 \prec \dots
    \prec j_k)$ of $P_f$ with respect to $\wt$. We need to show that $B$
    completely determines the max-slope arborescence $\arb^\wt_{P_f,\1}$.

    Let $v$ be a vertex of $P_f$ not contained in $P_f^\1 = B_f$ and let $I =
    I(v)$. It follows from the structure of polymatroid polytopes
    and~\eqref{eqn:MSarb} that
    \[
        \arb^\wt_{P_f,\1}(v) \ = \ \argmax \Bigl\{ \frac{\inner{\wt}{u-v}}{\1(u-v)} : u
        \in \Nb{P_f,\1}{v} \Bigr\}  \ = \ v + (f(I \cup j) - f(I)) e_j \, ,
    \]
    where $j$ is minimal with $j \notin \ov{I}$. Now
    Lemma~\ref{lem:mp_implies_pr} implies that $j = j_i$, where $i$ is minimal
    with $j_i \not \in \ov{I}$. This shows the claim.
\end{proof}

\newcommand\Ch{\mathrm{CH}}%
Thus, describing the monotone path polytope is equivalent to describing the
max-slope pivot polytope. Theorem~\ref{mainthm:ppnormmp} also implies that
$\PP_{P_f,\1}$ is a generalized permutahedron. In fact, we can give a nice
presentation as a Minkowski sum of standard simplices.  For any $S \subseteq
E$, we define the standard simplex $\Delta_S = \conv( e_s : s \in S)$.
For a flat $F \in \LF(f)$, let us define
\[
    \rho(F) \ := \  \sum_{\ov{\emptyset} = F_0 \subset \cdots \subset F_l
    \subset F} \prod_{i=1}^{l-1} |F_i\setminus F_{i-1}| \, ,
\]
where the sum is over all saturated chains in $\LF(f)$ ending in $F$.

\begin{prop}\label{prop:pivpolymat}
    Let $(E,f)$ be a polymatroid with lattice of flats $\LF(f)$.  Then
    \[
        \PP_{P_f,\1} \ = \ 
        \sum_{F \in \LF(f) \setminus\{ E\}}  \rho(F)\,\Delta_{E \setminus F} \, .
    \]
\end{prop}

\begin{proof}
    For a vertex $v \in P_f$ not contained in $B_f$, we infer
    from~\eqref{eqn:PPv} that
    \[
        \PP_{P_f,\1}(v) \ := \ \conv\bigl\{ (f(I(v) \cup j) - f(I(v))) \, e_j
        : j \not \in I(v) \bigr\} \, .
    \]
    Now $ f(I(v) \cup j) - f(I(v)) > 0$ if and only if $j \not \in \ov{I(v)}$.
    In particular $\PP_{P_f,\1}(v)$ only depends on the flat $F = \ov{I(v)}$.
    For every vertex $v$ with $F = \ov{I(v)}$ there is a unique chain of flats
    $\ov{\emptyset} = F_0 \subset \cdots \subset F_l \subset F$ and $i_s \in
    F_s \setminus F_{s-1}$. Thus, the number of vertices with $\ov{I(v)} = F$
    is precisely $\rho(F)$. The representation then follows
    from~\eqref{eqn:PPsum}.
\end{proof}

A nonempty collection $\Build \subseteq 2^E$ is a \Def{building
set}~\cite{GenPermOrig} if for all $S,T \in \Build$
\[
    S \cap T \neq \emptyset 
    \quad \Longrightarrow \quad
    S \cup T  \in  \Build \, .
\]
Let $y_S \in \R_{>0}$ for all $S \in \Build$.  The generalized permutahedron
\[
    \Delta(\Build) \ : = \ \sum_{S \in \Build} y_S \Delta_S
\]
is called a \Def{nestohedron}. Building sets and nestohedra were introduced by
Postnikov~\cite{GenPermOrig} and independently by
Feichtner--Sturmfels~\cite{nestsets}. In~\cite{GenPermOrig}, the definition of
building sets requires $\{i\} \in \Build$ for every $i \in E$. This only adds a
translation by $e_E = (1,\dots,1)$ but is quite handy for the combinatorial
description of $\Delta(\Build)$. We leave it out for the following reason.

\begin{prop}
    Let $\emptyset \neq \UC \subseteq 2^E$ be a \Def{union-closed} family of
    sets, that is, $S \cup T \in \UC$ for all $S,T \in \UC$. Then $\UC$ is a
    building set.
\end{prop}

Edmonds~\cite[Theorem 27]{edmondsorig} showed that $\LF(f) \subseteq 2^E$ is
closed under intersections.	We define for a polymatroid $(E,f)$
\[
    \UC(f) \ := \ \{ E \setminus F : F \in \LF(f) \} \, .
\]

\newcommand\NSC{\mathcal{N}}%
\newcommand\Nest{N}%
Let $\Build \subseteq 2^E$ be a building set. A \Def{nested set} is a subset
$\Nest \subseteq \wBuild := \Build \cup \binom{E}{1}$ such that 
\begin{enumerate}[\rm (N1)]
    \item For any $S,T \in \Nest$, we have $S \subseteq T$, $T \subseteq S$, or
        $S \cap T = \emptyset$;
    \item For any $S_1,\dots,S_k \in \Nest$ with $k \ge 2$ if $S_1 \cup \cdots
        \cup S_k \in \wBuild$, then $S_i \cap S_j \neq \emptyset$ for some $i
        < j$;
    \item If $S \in \wBuild$ is inclusion-maximal, then $S \in \Nest$.
\end{enumerate}

The collection $\NSC(\Build)$ of nested sets of $\Build$ is called the
\Def{nested set complex}.

\begin{prop}[{\cite[Thm.~7.4]{GenPermOrig}}]\label{prop:nested_set_complex}
    Let $\Build \subseteq 2^E$ be a building set. Then the face lattice of
    $\Delta(\Build)$ is anti-isomorphic to the nested set complex
    $\NSC(\Build)$. In particular $\Delta(\Build)$ is a simple polytope.
\end{prop}

\begin{proof}[Proof of Theorem~\ref{Imainthm:ppnesto}]
    Theorem~\ref{mainthm:ppnormmp} shows that $\MPP{P_f}$ is normally
    equivalent to $\PP_{P_f,\1}$.  It now follows from
    Proposition~\ref{prop:pivpolymat} that $\PP_{P_f,\1}$ equals
    $\Delta(\UC(f))$ for $y_{E\setminus F} = \rho(F)$ for all $F \in \LF(f)$.
    Since nestohedra are simple, this holds true for $\PP_{P_f,\1}$ as well as
    for $\MPP{P_f}$.
\end{proof}

\begin{cor}\label{cor:flag_simple}
    For every polymatroid $(E,f)$, the monotone path polytope $\MPP{f}$ as
    well as the max-slope pivot polytope $\PP_{P_f,\1}$ are simple polytopes.
\end{cor}

The facial structure of a nestohedron is determined by the maximal nested sets
of $\Build$. Postnikov~\cite{GenPermOrig} gave a nice description in terms of
certain rooted forests. We encode a rooted forest $\Forest$ on $E$ by the map
$\dForest : E \to 2^E$ such that $\dForest(i)$ is the collection of nodes
(including $i$) in the subtree rooted at $i$. That is, $\dForest(i)$ are the
descendants of $i$. Two nodes $i,j$ are \Def{comparable} if $\dForest(i)
\subseteq \dForest(j)$ or $\dForest(j) \subseteq \dForest(i)$.

For a building set $\Build$, a \Def{B-forest} is a rooted forest $\Forest$ on
$E$ such that 
\begin{enumerate}[\rm (F1)]
    \item $\dForest(i) \in \wBuild$ for all $i \in E$;
    \item If $s_1,\dots,s_k \in E$ for $k \ge 2$ satisfy $\bigcup_j
        \dForest(s_j) \in \wBuild$, then $s_i,s_j$ are comparable for
        some $i < j$;
    \item For every inclusion-maximal $S \in \wBuild$ there is $i \in E$ with
        $\dForest(i) = S$.
\end{enumerate}
The maximal nested set corresponding to a B-tree is $\{ \dForest(i) : i
\in E \}$.

\begin{prop}\label{prop:UC-Btrees}
    Let $\UC \subseteq 2^E$ be a union-closed family such that $\bigcup\UC =
    E$. The B-forests are in bijection to a collection $(t_i, S_i)$ for
    $i=1,\dots,k$ such that $S_1 \subset S_2 \subset \cdots \subset S_k = E$
    is a chain in $\UC$, $t_i \in S_i \setminus S_{i-1}$ with $S_0 :=
    \emptyset$, and for any $i \ge 0$ and nonempty $R \subseteq E \setminus (S_i
    \cup \{ t_{i+1},\dots,t_k\})$ it holds that $S_i \cup R \not\in \UC$.
\end{prop}
\begin{proof}
    Let $\Forest$ be a B-forest for $\UC$. Since $E \in \UC$, it
    follows that $\Forest$ is a tree. If $s_1,s_2 \in E$ are not leaves, then
    $\dForest(s_j) \in \UC$ for $j=1,2$ and $\dForest(s_1) \cup \dForest(s_2)
    \in \UC$ implies that $s_1$ and $s_2$ are comparable.  It follows that
    every node has at most one non-leaf child. Let $t_1,\dots,t_k$ be the
    non-leaves and set $S_i := \dForest(t_i)$. Then $S_1 \subset \cdots \subset
    S_k$ is a chain in $\UC$. The leaves are $L = \bigcup_{i \ge 1} S_i
    \setminus (S_{i-1} \cup t_i)$ and $T$ of nodes is incomparable iff $T
    \subseteq L \cup \{t_i\}$ for some $i$ and $T \setminus t_i \subseteq L
    \setminus S_i$; see also Figure~\ref{fig:Btree}. This shows that $t_i \in
    S_i$ for $i=1,\dots,k$ satisfies the condition and it is straightforward
    to check that every such collection yields a B-tree.
\end{proof}

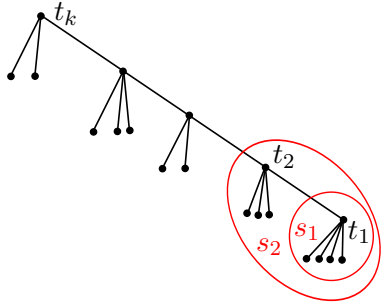
\begin{figure}
    \centering
    \begin{tikzpicture}
        \draw (0, 0) node[inner sep=0] 
        {\includegraphics[width=5cm]{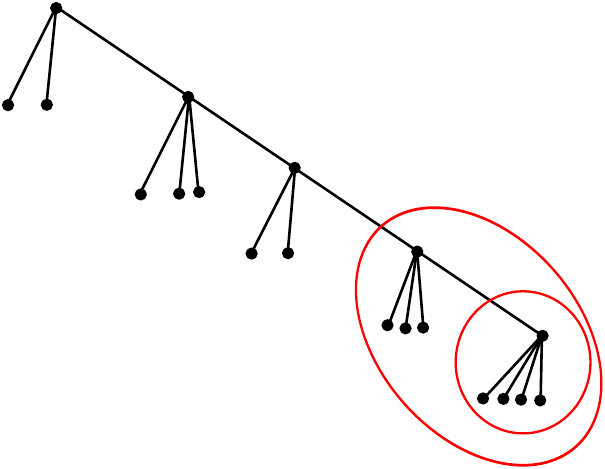}};
        \draw (-1.7, 1.9) node {$t_k$};
        \draw (1.2,-.0) node {$t_2$};
        \draw (2.2,-1) node {$t_1$};
        \draw (1.5,-1) node {\color{red}$s_1$};
        \draw (1,-1.2) node {\color{red}$s_2$};
    \end{tikzpicture}
    \caption{Bijection between B-trees and marked chains for union-closed
    families.}\label{fig:Btree}

\end{figure}

For the union closed family $\UC(f) = \{ E \setminus F : F \in \LF(f)\}$
associated to a polymatroid $(E,f)$ Proposition~\ref{prop:UC-Btrees} recovers
the greedy paths. Every chain $\emptyset =  S_0 \subset S_1 \subset \cdots
\subset S_k = E$ corresponds to a maximal chain of flats $\emptyset = F_0
\subset \cdots \subset F_k =E$ with $F_i = E \setminus S_{k-i}$ and $t_i \in F_i \setminus
F_{i-1}$. From this description, we can also deduce adjacency.

\begin{prop}
    Let $(j_1,\dots,j_k)$ be a greedy path for $(E,f)$ and $v \in \MPP{f}$ the
    corresponding vertex. The neighbors of $v$ correspond to the greedy paths
    $(j_1,\dots,j_t',\dots,j_k)$ for some $1 \le t \le k$ and
    $j_t' \in \ov{\{j_1,\dots,j_t\}} \setminus (\ov{\{j_1,\dots,j_{t-1}\}}
    \cup \{j_t\})$ or to greedy paths derived from the sequences 
    $(j_1,\dots,j_{s+1},j_s,\dots,j_k)$ for some $1
    \le s < k$.
\end{prop}
\begin{proof}
    Let $v$ be the vertex of $\MPP{f}$ corresponding to the greedy path
    $(j_1,\dots,j_k)$ and let $\Forest$ be the associated B-tree. For a weight
    $\wt \in \R^E$ it follows from Proposition~7.10 of~\cite{GenPermOrig} that
    $v \in \MPP{f}^\wt$ if and only if $\wt_i \ge \wt_j$ for all $i,j \in E$
    with $j \in \dForest(i)$. That is, $\wt$ is an order preserving map from
    the poset $\Forest$ into the real numbers. The cone of such $\wt$ is
    simplicial. The facets of the cone are given by the edges of the B-tree
    and correspond to neighbors of $v$ in $\MPP{f}$. The description of
    B-trees above now yields the claim. 
\end{proof}

Let us finally note that although Theorem~\ref{mainthm:ppnormmp} states that
$\PP_{P_f,\1}$ and $\MPP{f}$ are normally equivalent, they are not homothetic in
general.
	
\begin{prop}\label{prop:MPP_not_nesto}
    $\MPP{f}$ is in general not a sum of simplices. In particular, $\MPP{f}$
    is not necessarily a nestohedron.
\end{prop}
\begin{proof}
    Let $f$ be the rank function of the uniform matroid $U_{n,k}$, that is,
    $f(I) = \min(|I|,k)$. The polymatroid polytope $P_f$ is the convex hull of
    all $u \in \{0,1\}^n$ with $\sum_i u_i \le k$. The monotone path polytope
    satisfies
    \[
        \MPP{f} \ = \ 2\Delta_{n,1} + 2\Delta_{n,2} + \cdots + 2
        \Delta_{n,k-1} + \Delta_{n,k} \, .
    \]
    This is the permutahedron for the point $v = (0,\dots,0,1,3,\dots,2k-1)$;
    see also~\cite{Hypersimps}. Assume that there are $y_I \ge 0$ for all
    $I \subseteq [n]$ such that
    \[
        \MPP{f} \ = \ \sum_{I} y_I \Delta_I \, .
    \]
    Note that the left-hand side is invariant under the symmetric group.
    Hence, we can symmetrize to get
    \[
        \MPP{f} \ = \ \sum_{j=1}^n c_j S_j \quad \text{ where }
        \quad S_j \ := \ \sum_{I : |I| = j} \Delta_I 
    \]
    and $c = (c_1,\dots,c_n) \ge 0$. The vertex maximizing the right-hand side
    for the linear function \mbox{$\wt = (1,2,\dots,n)$} is given by $Mc$, where
    $M_{ij} = \binom{i-1}{j-1}$. In particular $c = M^{-1}v$. For $n = 4$ and
    $k=3$, we get $v = (0,1,3,5)$ and $c = (0,1,1,-1)$.
\end{proof}

\section{Flag Matroids}\label{sec:flag}

Let $M = (E,\Ind)$ be a matroid of rank $r$ and let $\kk = (k_1,\dots,k_s)$ be
a vector of integers satisfying $0 \le k_1 < k_2 < \cdots < k_s \le r$. The
\Def{flag matroid} $\FlagM_M^\kk$ of $M$ of rank $\kk$ is the
collection of chains 
\[
    I_* \ : \ I_1 \ \subset \ I_2 \ \subset \ \cdots \ \subset \ I_s
\]
of independent sets of $M$ with $|I_j| = k_j$ for $j=1,\dots,s$. Borovik,
Gelfand, Vince, and White~\cite{flagmats} introduced flag matroids more
generally in terms of strong maps. In this paper, we only treat the special
case of flag matroids of a matroid $M$. We refer
to~\cite{CoxeterMat} for relations to Coxeter matroids and to
Section~\ref{sec:toric} for the algebro-geometric point of view.  We call
$\FlagM_M := \FlagM_M^{(0,1,\dots,r)}$ the \Def{underlying flag matroid}
of $M$.

For a flag $I_*$, define $\delta(I_*)  :=  e_{I_1} + e_{I_2} + \cdots +
e_{I_s} \in \Z^E$ and with it the \Def{flag matroid polytope}~\cite{flagmats}
\[
    \Delta(\FlagM_M^\kk) \ := \ \conv \{ \delta(I_*) : I_* \in \FlagM_M^\kk \}
    \, .
\]
In this section, we relate flag matroid polytopes and monotone path polytopes
of matroids via a generalization of the independence polytope.

For $\kk = (k_1,\dots,k_s)$ define the \Def{rank-selected independent sets}
\[
    \Ind^\kk \ := \ \{ I \in \Ind: |I| = k_i \text{ for some } i=1,\dots,s \}
\]
and the \Def{rank-selected independence polytope}
$P_M^\kk  :=  \conv \{ e_I : I \in \Ind^\kk \}$.

\begin{lemma}\label{lem:inc_edges}
    Let $I,J \in \Ind^\kk$ with $|I| < |J|$. Then $[e_I,e_J]$ is an edge of
    $P_M^\kk$ if and only if $I \subset J$ and $|I|=k_i$, $|J| = k_{i+1}$ for
    some $1 \le i < s$.
\end{lemma}
\begin{proof}
    Assume that $[e_I,e_J] = (P_M^\kk)^\wt$ for some $\wt \in \R^E$. Let
    $\emptyset = I_0 \subset I_1 \subset \cdots \subset I_r$ be the sequence
    of independent sets obtained from the greedy algorithm on $P_M$ with
    respect to $\wt$. Let $|I|=k_i$ and $|J|=e_j$ with $i < j$. Since $e_I$ is
    the unique maximizer over the base polytope of the restriction $M_{k_i}$,
    we have $I_{k_i} = I$ and likewise $I_{k_j} = J$.  Now, since
    $\inner{\wt}{e_I} = \inner{\wt}{e_J} = \wt(I) + \wt(J \setminus I)$, it
    follows that $\wt(K) = \wt(I)$ for all $I \subseteq K \subseteq J$. Hence,
    $j = i+1$.  For the converse, take the linear function $\wt = e_I -
    e_{E\setminus J}$.
\end{proof}

In the same way as in Section~\ref{sec:MPP}, one shows that every cellular
string of $P_M^\kk$ is coherent.

\begin{theorem}
    Let $M = (E,\Ind)$ be a matroid and $\kk$ a rank vector. Then every
    $\1$-cellular string of $P^\kk_M$ is coherent. 
\end{theorem}

Lemma~\ref{lem:inc_edges} in particular implies that $\1$-monotone paths on
$P_M^\kk$ are precisely the elements of the flag matroid $\FlagM_M^\kk$.

\begin{theorem}\label{thm:FlagDelta}
    Let $M = (E,\Ind)$ be a matroid of rank $r$ and $\kk$ a rank vector.  The
    monotone path polytope $\MPP{P_M^\kk}$ is normally equivalent to the flag
    matroid polytope $\Delta(\FlagM_M^\kk)$.
\end{theorem}
\begin{proof}
    Note that the distinct values of the linear function $\1$ on the vertices
    of $P_M^\kk$ are precisely $k_1 < k_2 < \cdots < k_s$. For $i =
    1,\dots,s$, the fiber $\{ x: \1(x) = k_i\} \cap P_M^\kk$ is the base
    polytope of the truncation $M_{k_i}$ that we denote by $B_{k_i}$. It
    follows from Proposition~\ref{prop:MPPnorm} that $\MPP{P_M^\kk}$ is
    normally equivalent to
    \[
        B_{k_1} + B_{k_2} + \cdots + B_{k_s} \, .
    \]
    This is precisely the decomposition of $\Delta(\FlagM_M^\kk)$ given in
    Corollary~1.13.5 of~\cite{CoxeterMat}.
\end{proof}

It seems likely that the obvious generalization of rank-selected independence
polytopes to the setting of general flag matroids~\cite{flagmats} will
generalize Theorem~\ref{thm:FlagDelta}.

\begin{remark}
    For a rank-$r$ matroid $M$ with rank function $r_M$,
    Theorem~\ref{thm:FlagDelta} states that the base polytope of the flag
    polymatroid $\Mf{r}_M$ is normally equivalent to the base polytope of the
    underlying flag matroid $\FlagM_M$. This gives another justification for
    calling $\Mf{f}$ a \emph{(underlying) flag polymatroid}:
    \Cref{mainthm:ppnormmp} and~\Cref{Imainthm:ppnesto} imply that the facial
    structure of $P_{\Mf{f}}$ only depends on the flags of $\LF(f)$.  

    This prompts the question as to a notion of \emph{partial} flag
    polymatroid.  The rank vectors of flag matroids are subsets of the values
    $\{ \1(v) : v \in V(P_M) \} = \{  r_M(A) : A \in \LF(M) \}$. The important
    property for the description of flag matroid polytopes is that for every
    flat $A \in \LF(M)$ the vertices of $P_f \cap \{x : \1(x) = f(A)\}$ are
    vertices of $P_f$. This happens if and only if there are no \emph{long}
    edges: If $[u,v] \subset P_f$ is an edge with $\1(u) < \1(v)$ then $f(A)
    \le \1(u)$ or $\1(v) \le f(A)$ for all flats $A \in \LF(f)$. Note that the
    greedy algorithm implies $\1(u) = f(I(v)) = f(\ov{I(v)})$. The next result
    implies that this is characteristic for matroids.
\end{remark}

\newcommand\cov{\prec\!\mathrel{\raisebox{1pt}{$\scriptscriptstyle\bullet$}}}%
For flats $A, B \in \LF(f)$, we write $A \cov B$ if $A$ is covered by $B$,
that is, if $A \subset B$ and there is no flat $C$ with $A \subset C \subset
B$.

\begin{prop}
    Let $(E,f)$ be a polymatroid such that for all closed sets $A,B,C \in
    \LF(f)$ with $A \cov B$ we have $f(B) \le f(C)$ or $f(C) \le f(A)$, then
    $f$ is a multiple of a matroid rank function.
\end{prop}

\begin{proof}
    For any $A \in \LF(f)$ choose $B \in \LF(f)$ with $A \cov B$ and $f(B)$
    minimal. If $A \cov B'$, then $f(A) < f(B) \le f(B')$ and the condition
    implies that $f(B) = f(B')$. Now, for $A \cov B \cov C$ and $A \cov B'
    \cov C'$ with $f(C) < f(C')$, our condition implies $f(C) \le f(B') = f(B)
    < f(C)$. Thus $f(C) = f(C')$. Iterating the argument then shows that given
    two maximal chains if
    $A_0 \cov A_1 \cov \cdots \cov A_k$ and $A'_0 \cov A'_1 \cov \cdots \cov
    A'_l$, we have $f(A_i) = f(A_i')$ for all $i$ and, in
    particular, $k=l$. This implies that $\LF(f)$ is a graded poset. Assuming
    that $\{i\}$ is closed for every $i \in E$, we can scale $f$ so that
    $f(\{i\})=1$ for all $i \in E$. This implies that for $A \in \LF(f)$,
    $f(A) \in \Znn$ and submodularity shows
    \[
        f(A) \ \le \ \sum_{i \in A} f(\{i\}) \ = \ |A| \, . \qedhere
    \]
\end{proof}

\begin{example}[$S$-hypersimplices]
    Let $M$ be the uniform matroid $U_{n,n}$ on $n$ elements for which every
    subset $I \subseteq E$ is independent. The independence polytope $P_M$ is
    the unit cube and $\LF(M) = 2^E$ is the Boolean lattice. The base polytope
    of a truncation of $M$ to $k$ is the $(n,k)$-hypersimplex, that is, the
    convex hull of all $v \in \{0,1\}^E$ with $\sum_i v_i = k$.

    For $0\le k_1  < \cdots < k_s \le n$, the rank-selected independence
    polytope $P_M^\kk$ is an $S$-hypersimplex~\cite{Hypersimps} with $S = \{k_1,
    \dots, k_s \}$. For $k_i = i$, these are also the line-up polytopes for the
    cube introduced in \cite[Sect.~6.2.2]{lineuppolys}.  The corresponding
    monotone path polytope $\MPP{P_M^\kk}$ is homothetic to the permutahedron
    $\Pi(s,\dots,s,s-1,\dots,s-1,\dots,1,\dots,1)$ with multiplicities given
    by
    $k_1,\dots,k_s-k_{s-1}$.  The proof of
    Proposition~\ref{prop:MPP_not_nesto} shows that these need not be
    nestohedra.
\end{example}

The facial structure of $\Delta(\FlagM_M^\kk)$ and hence of $\MPP{P_M^\kk}$ is
given in Exercise~1.14.26 of~\cite{CoxeterMat}. For the underlying flag matroid
we can give an alternative description.  Recall that a set $K \subseteq E$ is
a \Def{cocircuit} of $M$ if it is inclusion-minimal with the property that it
meets every basis of $M$.

\begin{cor}
    For any matroid $M$, the flag matroid polytope $\Delta(\FlagM_M)$
    is a simple polytope normally equivalent to a nestohedron for the building
    set 
    \[
        \UC(M) \ = \ \{ K_1 \cup \cdots \cup K_m : m \ge 0, K_1,\dots, K_m
        \text{ cocircuits} \} \, .
    \]
\end{cor}
\begin{proof}
    It follows from \Cref{mainthm:ppnormmp}, \Cref{Imainthm:ppnesto},
    and \Cref{thm:FlagDelta} that $\Delta(\FlagM_M)$ is normally equivalent to
    the nestohedron for the union-closed family of sets $E
    \setminus F$ where $F$ ranges over all flats of $M$. Now $F$ is a flat if
    and only if $E\setminus F$ is a union of cocircuits
    \cite[Ex.~2.1.13(a)]{Oxley}.
\end{proof}

We close this section with a few thoughts on the max-slope pivot polytopes of
rank-selected independence polytopes. If $I$ is an independent set of rank
$|I| = k_i$ for $i < s$, then Lemma~\ref{lem:inc_edges} yields that the
$\1$-improving neighbors correspond to independent sets $J$ with $I \subset J$
and $|J| = k_{i+1}$. From~\eqref{eqn:PPsum}, we infer that 
\[
    \PP_{P_M^\kk,\1}(e_I) \ = \ \tfrac{1}{k_{i+1}-k_{i}} \conv\{ e_{J \setminus
    I} : I \subset J \in \Ind, |J| = k_{i+1} \} \, .
\]
The independent sets of the contraction $M/I$ are precisely those $K \subseteq
E \setminus I$ with $I \cup K$ independent in $M$. Hence
\[
    (k_{i+1}-k_{i}) \cdot 
    \PP_{P_M^\kk,\1}(e_I) \ = \ B_{(M/I)_{k_{i+1}-k_{i}}} \, ,
\]
where $(M/I)_{k_{i+1}-k_{i}}$ is the contraction of $M/I$ to rank
$k_{i+1}-k_{i}$. Consequently, the max-slope pivot polytope $\PP_{P_M^\kk,\1}$
is normally equivalent to 
\[
    \sum_{i=1}^{s-1} \sum_{\substack{F \in \LF(f)\\ \rk(F) = k_i}}
    B_{(M/F)_{k_{i+1}-k_{i}}} \, .
\]
If $k_{i+1} = k_i + 1$, then $B_{(M/F)_{k_{i+1}-k_{i}}}$ is the convex hull of
all $e_j$ such that $I \cup j \in \Ind$ and hence a standard simplex. This
prompts a generalization of  nestohedra where standard simplices are replaced
by matroid base polytopes. 

It is still true that $\Delta(\FlagM_M^\kk)$ is a weak Minkowski summand of
$\PP_{P_M^\kk,\1}$ but normal equivalence does not hold in general. We suspect
that the refinement of the normal cone of $\Delta(\FlagM_M^\kk)$ corresponding
to a flag $I_*$ reflects the freedom of the greedy algorithm to order the
elements in $I_{j+1} \setminus I_j$.

\section{Toric varieties in Grassmannians and flag varieties}\label{sec:toric}

In this section, we give a toric perspective on the monotone path polytopes of
realizable polymatroids and the relation between Grassmannians and flag
varieties.

For $1 \le r \le n$, let $\Gr{n,r}$ be the Grassmannian of $r$-dimensional
linear subspaces in $\C^n$. We can view a point $L \in \Gr{n,r}$ as the
rowspan of a full-rank matrix $A \in \C^{r \times n}$. The algebraic torus
$\T^n = (\C^*)^n$ acts on $\Gr{n,r}$ as follows.  If $L$ is represented by $A
= (a_1,\dots,a_n)$ and $\t = (t_1,\dots,t_n) \in \T^n$, then $\t$ sends $L$ to
$\t \cdot L = \rspan(\t\cdot A)$, where $\t \cdot A =
(t_1a_1,t_2a_2,\dots,t_na_n)$. 
The fixed points of this action are precisely the $r$-dimensional coordinate
subspaces of $\C^n$. 
In its Pl\"ucker embedding, a subspace $L$ is identified with its Pl\"ucker
vector $p(L) \in \Proj( \bigwedge^k\C^n) \cong \Proj^{\binom{n}{r}-1}$ with
$p(L)_J = \det(A_J) = \det(a_{j_1},\dots,a_{j_r})$, where $J = \{j_1 < \cdots
< j_r\}$ is an ordered $r$-subset of $[n]$. The fixed points then correspond
to Pl\"ucker vectors $p$ of the form $p_{J_0} \neq 0$ for a fixed $r$-subset
$J_0$ and $p_J =0$ otherwise.

The moment map $\mu : \Gr{n,r} \to \R^n$ of the action of $\T^n$ on $\Gr{n,r}$
is given by
\[
    \mu(L)_j \ = \ 
        \frac{\sum_{j \in J} |p(L)_J|^2}{\sum_{J} |p(L)_J|^2} \, ,
\]
where $J$ ranges over all $r$-subsets of $[n]$;
see~\cite[Sect.~2.1]{GGMS}. The image of $\Gr{n,r}$ under $\mu$ is
precisely the $(n,r)$-hypersimplex $\Delta(n,r)$, whose vertices correspond to
the fixed points.

Let $M = M(L) = ([n],\Ind)$ be the rank-$r$ matroid with $I \in \Ind$ if and
only if $(a_i)_{i \in I}$ is linearly independent. Note that this only depends
on $L$ and not on $A$.

\begin{theorem}[{\cite[Sect.~2.4]{GGMS}}]
    Let $L \in \Gr{n,r}$ be a subspace with matroid $M$. The Zariski closure
    of $\T^n \cdot L$ is a projective toric variety in $\Gr{n,r}$ with moment
    polytope $B_M$.
\end{theorem}

The independence polytope can also be obtained as a moment polytope. Choose a
representation $A$ of $L$ such that $e_1,\dots,e_r$ is in general position
with respect to $a_1,\dots, a_n$. That is, every linearly independent
collection $(a_i : i \in I)$ can be completed to a basis of $\C^r$ by any
choice of $r - |I|$ vectors from $e_1,\dots,e_r$.  Define $\widehat{A} :=
(A,E) = (a_1,\dots,a_n,e_1,\dots,e_r) \in \C^{r \times (n+r)}$ and
$\widehat{L} := \rspan(\widehat{A})$.  We can view $\T^n$ as a subtorus of
$\T^{n+r}$ acting on $\widehat{L}$ by 
\[
    \t \cdot \widehat{A} \ = \ (t_1a_1,\dots,t_na_n,e_1,\dots,e_r) \, .
\]

\begin{cor}\label{cor:toric_indep}
     The Zariski closure of the orbit of $\widehat{L} = \rspan(\widehat{A})$
     under $\T^n$ is a projective toric variety $X_{\widehat{L}} \subseteq
     \Gr{n+r,r}$ with moment
     polytope $P_M$.
\end{cor}
\begin{proof}
    Let $\lambda^\wt(t) = (t^{\wt_1},\dots,t^{\wt_n})$ be a one-parameter
    subgroup. On the level of Plücker vectors, it can be seen that $\lim_{t
    \to \infty} \lambda^\wt(t) \cdot \widehat{L}$ is fixed by $T^n$ if and
    only if there is a unique $I \in \Ind$ such that $\wt(I) = \sum_{i \in I}
    \wt_i$ is maximal. Let $p$ be the Plücker vector of the limit point for
    some $I \in \Ind$.  Then $p_J \neq 0$ if and only if $I \subseteq J$ and
    $J\setminus I \subseteq \{n+1,\dots,n+r\}$. The representation of the
    moment map above yields $\mu(p) = e_I$ and shows $\mu(X_{\widehat{L}}) =
    P_M$.
\end{proof}

\newcommand\GL{\mathrm{Gl}}%
For $1 \le r \le n$, let $\Fl{n,r}$ be the \Def{flag variety} of complete
flags $0 = F_0 \subset F_1 \subset \cdots \subset F_r \subseteq \C^n$ with
$\dim F_i = i$ for $i = 1,\dots,r$. Any such flag can be represented by a
full-rank matrix $A \in \C^{r \times n}$. If $A_i \in \C^{i \times n}$ is the
submatrix obtained from $A$ by taking the first $i$ rows, then $F_i =
\rspan(A_i)$ for $i=0,\dots,r$ defines a complete flag $F_{\bullet} =
(F_i)_{i=0,\dots,r}$. If $A$ and $A'$ define the same flag, then $A' = gA$,
where $g \in B \subset \GL(\C^r)$,  the (standard) Borel subgroup of
invertible lower-triangular matrices.

Notice that $M(F_{i-1})$ is a quotient of $M(F_i)$ and Theorem~1.7.3
of~\cite{CoxeterMat} asserts that $(M(F_1), M(F_2), \dots, M(F_r))$ is a
general flag matroid. We call the flag $F_\bullet$ \Def{very general} if
$M(F_i)$ is the $i$-th truncation of $M(F_r)$ for each $i=1,\dots,r-1$. 

$\Fl{n,r}$ is naturally a subvariety of $\prod_{i=1}^r \Gr{n,i}$ and the
diagonal action of $\T^n$ extends to $\Fl{n,r}$. Thus, any flag $F_\bullet$
yields a toric subvariety
\[
    Y_{F_\bullet} \ \subseteq \ X_{F_1} \times X_{F_2} \times \cdots \times
    X_{F_r} \, .
\]
Theorem 6.19 in \cite{FlagMatAlgGeo} asserts that the moment polytope of
$Y_{F_\bullet}$ is $B_{M(F_1)} + B_{M(F_2)} + \cdots + B_{M(F_r)}$, the
polytope of the flag matroid $(M(F_1), M(F_2), \dots, M(F_r))$.

\begin{theorem}\label{thm:genflag_uflagmat}
    If $F_\bullet$ is very general, then the moment polytope of
    $Y_{F_\bullet}$ is $\Delta(\FlagM_M)$, where $M = M(F_r)$. Conversely, if 
    $M$ is a rank-$r$ matroid realizable over $\C$, then there is very
    general flag $F_\bullet$ with $M(F_r) = M$.
\end{theorem}
\begin{proof}
    The first statement follows directly from the preceeding discussion and
    the definition of very general flag.  For the second statement, let $L =
    \rspan(A)$ be a realization of $M$ with $A \in \C^{r \times n}$.
    Projecting $L$ onto a general linear subspace $L' \subset L$ of dimension
    $r-1$ yields a realization of the first truncation of $M$. Iterating this
    yields a very general flag. Up to a change of coordinates this is means
    that the flag $F_\bullet$ associated to $gA$ for any general $g \in
    \GL(\C^r)$ is very general and since $L = \rspan(gA)$, this proves the
    claim.
\end{proof}

There is a rational map $\phi : \Gr{n+r,r} \dashrightarrow \Fl{n,r}$.  Let
$\widehat{L} \in \Gr{n+r,r}$ such that $\widehat{L}$ is represented by a matrix of the
form $\widehat{A} = (A,E)$, where $A \in \C^{r \times n}$ is of full rank and $E =
(e_1,\dots,e_r)$. Then $\phi$ takes $\widehat{L}$ to the flag $F_\bullet$ with
$F_i = \rspan(A_i)$ as above. The set of such $\widehat{L}$ is Zariski open and
$\phi$ is a rational surjective map. The fibers of $F_{\bullet}$ are
represented by $(gA,E)$ with $g \in B$.

Note that $\phi$ is equivariant with respect to the action of $\T^n$. 

\begin{prop}
    Let $\widehat{L} \in \Gr{n+r,r}$ such that
    $\phi(\widehat{L}) =  F_\bullet$ is defined.  Then $\phi$ is a regular map
    on $X_{\widehat{L}}$ with image $Y_{F_\bullet}$. The preimage of
    $F_{\bullet}$ in $X_{\widehat{L}}$ are the linear subspaces $\t
    \widehat{L}$, where $\t = (t,t,\dots,t) \in \T^n$. 
\end{prop}
\begin{proof}
    The flag variety $\Fl{n,r}$ is embedded in the projective space over
    $\bigoplus_{k=1}^r \bigwedge^k\C^n$ with coordinates $(p_K)_{K}$, where
    $K$ ranges over all non-empty subsets of $[n]$ of size $|K| \le r$.
    If $F_\bullet$ is represented by $A$,
    then it is represented by the flag minors $(p(F_\bullet)_K)_{K}$
    with $p(F_\bullet)_K = \det((A_k)_K)$, where $k = |K|$;
    see~\cite[Ch.~14.1]{MillerSturmfels}. Let $\widehat{L} =
    \rspan(\widehat{A})$, where $\widehat{A} = (A,E)$. On the level of Plücker
    vectors, the map $\phi$ is given by a coordinate projection: For $K
    \subseteq [n]$ and $|K| = k$, $p(\phi(\widehat{L}))_K =
    p(\widehat{L})_{K \cup \{n+1,\dots,n+k\}}$.
    Let $(gA,E)$ represent a preimage of $F_\bullet$. Then $(gA,E) = \t \cdot
    (A,E)$ if and only if $g$ is a multiple of the identity matrix.
\end{proof}

\newcommand\Fan{\mathscr{N}}%
Kapranov, Sturmfels, and Zelevinsky~\cite{KSZ} studied quotients of toric
varieties by subtori. Let $X \subset \Proj^{n-1}$ be a projective toric
variety with $n$-dimensional torus $\T$ and fan $\Fan$ in $\R^n$. A subtorus
$H \subset \T$ is represented by a rational subspace $U \subset \R^n$. Define
an equivalence relation on $\R^n/U$ by setting $q + U \sim q'+U$ if $q + U$
meets the same cones of $\Fan$ as $q' + U$. The equivalence classes form a fan
$\Fan/U$ in $\R^n/U$ called the \Def{quotient fan}. A toric variety $Y$ with
fan $\Fan/U$ is called a \Def{combinatorial quotient}. We can now state the
relationship between $X_{\widehat{L}}$ and $Y_{F_\bullet}$.

\begin{theorem}\label{thm:comb_quot}
    Let $F_\bullet$ be a very general flag. Then the toric variety
    $Y_{F_\bullet}$ is a combinatorial quotient for the action of $H$ on
    $X_{\widehat{L}}$.  Moreover, $Y_{F_\bullet}$ is a smooth toric variety.
\end{theorem}

In~\cite{KSZ}, the authors construct a canonical combinatorial quotient
associated to $X$ and $H$, called the \Def{Chow quotient} $X/\!\!/H$. This is
a toric variety associated to the Chow form of the closure of $H \cdot E_0$,
where $E_0$ is the distinguished point of $X$. The embedding $H \subset T$
yields a linear projection $\pi : \R^n \to U$.  Let $\Sigma_\pi(P)$ be the
fiber polytope~\cite{BSFiberPoly} of the pair $(P,\pi)$; see
also~\cite[Section 2]{KSZ}.  The following is a consequence of Theorem 2.1,
Proposition 2.3, and Lemma 2.6 of \cite{KSZ}.

\begin{theorem}
    Let $X$ be the toric variety associated to the lattice polytope $P$.  Then
    the Chow quotient $X/\!\!/H$ is the toric variety associated to the fiber
    polytope $\Sigma(P,\pi)$. 
\end{theorem}

\begin{proof}[Proof of Theorem~\ref{thm:comb_quot}]
    By Corollary~\ref{cor:toric_indep}, the polytope associated to
    $X_{\widehat{L}}$ is the independence polytope $P_M$ of the matroid $M =
    M(L)$ for $L = \rspan(A)$. The linear subspace associated to the subtorus
    $H$ is $U = \{ (u,u,\dots,u) : u \in \R \}$. The linear projection $\pi$
    is $\pi(x) = x_1 + \cdots + x_n$. Hence the fiber polytope $\Sigma(P,\pi)$
    is the monotone path polytope $\MPP{P_M}$. If $F_\bullet$ is very general,
    then the moment polytope of $Y_{F_\bullet}$ is $\Delta(\FlagM_M)$ by
    Theorem~\ref{thm:genflag_uflagmat}.  The first claim now follows from
    Theorem~\ref{thm:FlagDelta} and the fact that normally equivalent
    polytopes have the same underlying fan.

    As for the second claim, we note from Theorem~\ref{Imainthm:ppnesto} (see
    also Corollary~\ref{cor:flag_simple}) that $\MPP{P_M}$ is a simple
    generalized permutahedron. This implies that at every vertex, there are
    precisely $\dim \MPP{P_M}$ many incident edges and primitive vectors along
    the edge directions are of the form $e_i - e_j$ and hence provide a
    lattice basis.  This is equivalent to $Y_{F_\bullet}$ being smooth.
\end{proof}

We can extend this relation to realizable polymatroids; see end of
Section~\ref{sec:polymatroids}. Let $f : 2^{[n]} \to \Z_{\ge0}$ be a integral
polymatroid realized by linear subspaces $V_1,\dots,V_n \subset \C^r$ so that
$f(I) = \dim_\C \sum_{i \in I} V_i$. For $i=0,\dots,n$ define $s_i =
\sum_{j=1}^i \dim V_j$. We can represent $f$ by a full-rank matrix $A =
(a_1,\dots,a_{s_n}) \in \C^{r \times s_n}$ by letting
$a_{s_{i-1}+1},\dots,a_{s_i}$ be a basis of $V_i$.  Let $L_f = \rspan(A,E) \in
\Gr{s_n+r,r}$. We view $T^n$ as a subtorus of $T^{s_n+r}$ 
\[
    \T^n = \{ (t_1,\dots,t_1,t_2,\dots,t_2,\dots,t_n,\dots,t_n,1,\dots,1) :
    t_1,\dots,t_n \in \C^* \} \, .
\]
The same argument as before then shows 
\begin{theorem}\label{thm:X_f}
    Let $L_f \in \Gr{s_n+r,r}$ as above. The Zariski closure of the orbit
    $\T^n \cdot L_f$ is a projective toric variety $X_f$ with moment polytope
    $P_f$.
\end{theorem}

The matrix $A$ also defines a flag $F_\bullet \in \Fl{s_n,r}$ and a toric
variety $Y_{f}$ with respect to the action of $\T^n$. The constituents in
every $\Gr{n,i}$ are not so easy to describe as they depend on the choice of a
basis for each $V_i$. However, the relation between the toric
varieties stays intact and the same proof as for Theorem~\ref{thm:comb_quot}
yields the following.

\begin{theorem}\label{thm:comb_quot_poly}
    Let $F_\bullet$ be a very general flag,  Then the toric variety $Y_f$ is a
    combinatorial quotient for the action of $H$ on $X_f$ and the moment
    polytope of $Y_f$ is normally equivalent to $\MPP{f}$. In particular,
    $Y_f$ is a smooth toric variety for every realizable polymatroid.
\end{theorem}

\newcommand\ef{f'}%
\newcommand\eE{E'}%
\section{Independent Set Greedy Paths and partial permutahedra}\label{sec:indep_greedy}

The base polytope of the flag polymatroid $\MPP{f}$ of Section~\ref{sec:MPP}
is a polytope whose vertices encode the different greedy paths for optimizing
on the base polytope $B_f$. The greedy algorithm (Theorem~\ref{thm:greedy})
can also be used to optimize linear functions over $P_f$ by simply stopping
when $\wt_{\sigma(i)} < 0$. It turns out that up to a simple modification of
the polymatroid, the space of partial greedy paths may also be represented by
a flag polymatroid. We apply this extension to resolve a conjecture on partial
permutation polytopes of Heuer--Striker~\cite{partperms}.

For a polymatroid $f : 2^E \to \R$ with $E = [n]$, define $\ef : 2^{\eE}
\to \R$ with $\eE := [n+1]$ by 
\[
    \ef(A) \ := \ \begin{cases}
        f(A) & \text{if } n+1 \not \in A\\
        f(E) & \text{otherwise.}
    \end{cases}
\]

\begin{prop}
    Let $f$ be a polymatroid and $\ef$ as defined above. Then $(\eE,\ef)$ is a
    polymatroid with base polytope
    \[
        B_{\ef} \ = \ \{ (x,f(E)-\1(x)) : x \in P_f \} \ \cong \ P_f \, .
    \]
\end{prop}
\begin{proof}
    For $x \in \R^E$, let $x' := (x,f(E)-\1(x))$.  Let $B'_f := \{ x' : x \in
    P_f \}$, which is linearly isomorphic to $P_f$. Every edge of $B'_f$ is of
    the form $[u',v']$ where $[u,v] \subseteq P_f$ is an edge.  If $u-v =
    \mu(e_i - e_j)$ for some $\mu \neq 0$, then $\1(u)=\1(v)$ and hence $u'-v'
    = \mu(e_i - e_j)$. If $u - v = \mu e_i$, then $u'-v' = \mu (e_i -
    e_{n+1})$. Hence $B'_f = B_g$ is a generalized permutahedron or
    polymatroid base polytope for some polymatroid $g : 2^{\eE} \to \R$. For
    $A \subseteq \eE$ we have $g(A) = \max\{ \1_A(x) : x \in B_g\}$. If $n+1
    \not\in A$, then $g(A) = \max \{ \1_A(x) : x \in P_f\} = f(A) = \ef(A)$.
    If $A = S \cup \{n+1\}$, then we maximize $\1_A(x') = \1_{S}(x) + f(E) -
    \1(x) = f(E) - \1_{E \setminus S}(x)$ over $P_f$. Since $0 \in P_f
    \subseteq \Rnn^E$, this means that the maximal value is $f(E) = \ef(A)$.
    Hence $g = \ef$ and $B_{\ef} \cong P_f$.
\end{proof}

For a weight $\wt \in \R^E$, define $\tilde{\wt} := (\wt,0) \in \R^{\eE}$.
Then optimizing $\wt$ over $P_f$ is precisely the same as optimizing
$\tilde{\wt}$ over $P_{\ef}$, which can be done with the usual greedy
algorithm. If $(j_1,j_2,\dots,j_k)$ represents a partial greedy path on $P_f$,
then $(j_1,j_2,\dots,j_k,n+1)$ is the corresponding greedy path on $P_{\ef}$.

From the definition we get that
\[
    \LF(\ef) \ = \ (\LF(f) \setminus \{E\}) \cup \{ \eE \} \, .
\]

\begin{cor}
    Let $(E,f)$ be a polymatroid. The $\1$-monotone paths on $P_f$ from $0$ to
    some vertex are in bijection to $\1$-monotone paths on $P_{\ef}$ from $0$
    to a vertex of $B_{\ef}$. All these greedy paths are coherent and the
    monotone path polytope $\MPP{\ef}$ is normally equivalent to the
    nestohedron with building set
    \[
        \UC(\ef) \ = \ \{ (E \setminus F) \cup \{n+1\} : F \in \LF(f) \} \, .
    \]
\end{cor}

\newcommand\ParPerm{\mathcal{P}}%
As an application of these tools, we completely resolve a conjecture of Heuer
and Striker~\cite{partperms} about partial permutahedra. For $m,n \ge 1$ the
\Def{$(m,n)$-partial permutahedron} $\ParPerm(m,n) \subset \R^m$ is the convex
hull of all points $x \in \{0,1,\dots,n\}^m$ such that the non-zero entries
are all distinct. 

\begin{customconj}{5.24}[\cite{partperms}]
    Faces of $\ParPerm(m,n)$ are in bijection with flags of subsets of $[m]$
    whose difference between largest and smallest nonempty subsets is at most
    $n-1$. A face of $\ParPerm(m,n)$ is of dimension $k$ if and only if the
    corresponding flag has $k$ missing ranks. 
\end{customconj}
	
In their paper, they prove the case when $m = n$ via the observation that
$\ParPerm(m,n)$ is the graph associahedron for the star graph, the so-called
\emph{stellohedron}. Graph associahedra are in particular nestohedra, and they
use the nested set structure to verify the desired bijection. Their missing
link for the general case with $m \neq n$ was a lack of a nested set
structure. We resolve their conjecture by using our new tools to endow the
partial permutahedron with a nested set structure. 

Note that from the convex hull description, it is apparent that the polytope
$\ParPerm(m,n)$ is anti-blocking. The vertices of $\ParPerm(m,n)$ are the
points $v \in \R^m$ with $0 \le k \le \max(m-n,0)$ zero entries and the
remaining entries a permutation of $\{n,n-1,\dots,n-(m-k)+1)\}$. The face $F$
of $\ParPerm(m,n)$ that maximizes $\1$ is the convex hull of permutations of
$(0,\dots,0,1,\dots,n)$ if $n \le m$ and $(n-m+1,\dots,n)$ if $n > m$. Since
$F$ is a permutahedron and $\ParPerm(m,n) = \Rnn^m \cap (F - \Rnn^m)$, we
conclude that the $(m,n)$-partial permutahedron $\ParPerm(m,n)$ is a
polymatroid polytope.

For $n > m$, $\ParPerm(m,n)$ is normally equivalent to the polymatroid of the
permutahedron and hence combinatorially (even normally) equivalent to 
$\ParPerm(m,m)$. Thus the only relevant case is $m > n$.

For $1 \le n \le m$, let $U_{m,n}$ be uniform matroid on $[m]$ of rank $n$.
The partial greedy paths for $U_{m,n}$ are precisely sequences
$(j_1,j_2,\dots,j_k)$ with $j_1,\dots,j_k \in [m]$ distinct and $k \le n$. The
corresponding chain of flats is $\emptyset = A_0 \subset A_1 \subset \cdots
\subset A_k$, where $A_i = \{j_1,\dots,j_i\}$ for $i < k$. If $k < n$, then
$A_k = \{j_1,\dots,j_k\}$ and $A_k = [m]$ otherwise.

The rank function of $U_{m,n}$ is given by $r_{m,n}(A) = \min(|A|,n)$ and we
let $f_{m,n} := r'_{m,n}$ as defined above. That is, $f_{m,n} : 2^{[m+1]}
\to \Znn$ with $f_{m,n}(A) = \min(|A|,n)$ if $m+1 \not \in A$ and $= n$
otherwise.

Let 
\[
    \ParPerm'(m,n) \ = \ \{ (x,\tbinom{n+1}{2} - \1(x)) : x \in \ParPerm(m,n)
    \} \ \subset \ \R^{m+1} \, .
\]
be the embedding of $\ParPerm(m,n)$ into the hyperplane $\{ y \in \R^{m+1} : y_1+\cdots+y_{m+1} = 1+
\cdots + n \}$.

\begin{theorem}
    The partial permutahedron $\ParPerm'(m,n)$ is normally equivalent to the
    polymatroid polytope of the flag polymatroid $\MPP{f_{m,n}}$.
\end{theorem}
\begin{proof}
    Let $c \in \R^{m+1}$ be a general linear function. We show that
    $\MPP{f_{m,n}}^c$ is a vertex if and only if $\ParPerm'(m,n)^c$ is a
    vertex. Equivalently, we show that $c$ determines a $\1$-monotone path on
    $P_{f_{m,n}}$ from $0$ to some vertex $u \in P_{f_{m,n}}$ if and only if
    $\ParPerm'(m,n)^c$ is a vertex. Let $\sigma$ be a permutation of $[m]$
    such that $c_{\sigma(1)} \ge c_{\sigma(2)} \ge \cdots \ge c_{\sigma(k)} >
    c_{m+1} \ge c_{\sigma(k+1)} \ge  \cdots \ge c_{\sigma(m)}$.  Now, $c$
    determines a greedy path on $P_{f_{m,n}}$ if and only if $c_{\sigma(i)}
    \neq c_{\sigma(j)}$ for $1 \le i < j \le \min(k,n)$.

    The face $\ParPerm'(m,n)^c$ is linearly isomorphic to the face
    $\ParPerm(m,n)^{\tilde{c}}$ for the function $\tilde{c} = (c_1 - c_{m+1},
    c_2 - c_{m+1},\dots, c_m - c_{m+1})$.  From the definition of vertices of
    $\ParPerm(m,n)$, we see that $\ParPerm(m,n)^{\tilde{c}}$ is a vertex if
    and only if the same condition is satisfied.
\end{proof}

\begin{cor} \label{cor:ppermnest}
    For $m \ge n \ge 1$, the $(m,n)$-partial permutahedron is combinatorially
    isomorphic to the nestohedron $\Delta(\UC(m,n))$ for the union closed set
    \[
        \UC(m,n) \ = \  \{ S \cup \{m+1\} : S \subseteq [m], |S| > m-n \text{
            or } S = \emptyset \}  \, .
    \]
\end{cor}

With this description, we are able to prove Conjecture $5.24$. Our proof is
analogous to their proof of the case in which $m = n$. 
	
\begin{theorem}\label{thm:heuer_striker}
    Faces of $\mathcal{P}(m,n)$ are in bijection with flags of subsets of
    $[m]$ whose difference between largest and smallest nonempty subsets is at
    most $n-1$. A face of $\mathcal{P}(m,n)$ is of dimension $k$ if and only
    if the corresponding flag has $k$ missing ranks.	
\end{theorem}
\begin{proof}
    It suffices to describe the nested set complex. We first define a
    bijection between nested sets and flags. Let $\Nest \subseteq
    \Build$ be a nested set. Consider the set $S = \{X \in
    \Nest: m + 1 \in X\}$. Since a nested set contains the maximal
    element of the building set, $[m+1] \in \Nest$ meaning that $S$ is
    nonempty. Furthermore, $\Nest \setminus S$ must consist entirely of
    singletons, since every set in $\Build$ of size at least $2$ contains
    $m+1$. Let $S_{0} = \{y: \{y\} \in \Nest \setminus S\}$. Let $x \in
    S_{0}$, and let $T \in S$. Then, by the nested set axioms, either $x \in
    T$ or $T \cup x \notin \Build$. However, $|T \cup x| \geq |T| \geq
    m-n+1$ and $m+1 \in T \cup x$, so $T \cup x \in \Build$. Hence, for
    all $x \in S_{0}$, we must have that $x \in T$ for all $T \in S$. 
		
    With these observations about the nested set structure in mind, we are
    prepared to define our bijection. Note that any two sets in $S$ must
    intersect, since they all contain $m+1$, so by the nested set axioms, $S$
    must be a flag of subsets $S_{1} \subseteq S_{2} \subseteq \dots \subseteq
    S_{k} = [m+1]$ in $[m+1]$. Then the flag we associate to $\Nest$ is
    exactly 
    \[
        S_{1} \setminus (S_{0} \cup m+1) \subseteq S_{2}\setminus (S_{0} \cup
        m+1)  \subseteq \dots \subseteq S_{k} \setminus (S_{0} \cup m+1) =
        [m+1] \setminus (S_{0} \cup m+1). 
    \] 
    Note that $|S_{1}| \geq m-n+1$, so $|[m+1] \setminus S_{1}| \leq n-1$.
    Hence, this map is well-defined. To see it is a bijection, start with a
    chain $T_{1} \subseteq T_{2} \subseteq \dots \subseteq T_{k}$. Then the
    corresponding nested set is the following:
    \[
        \Nest = \{T_{i} \cup [m+1] \setminus T_{k}: i \in [k]\} \cup
    \{\{x\}: x \in [m] \setminus T_{k}\}.
    \] 
    Since $|T_{k} \setminus T_{1}| \leq n-1$, $|T_{1} \cup ([n] \setminus
    T_{k}))| \geq m - (n-1) = m-n+1$, which is precisely the condition
    necessary to ensure that each $T_{i} \cup ([n+1] \setminus T_{k})$ is in
    the building set. It remains to show that $\Nest$ satisfies the
    nested set axioms, but this is immediate since a flag of building sets and
    collection of singletons contained in every set will always satisfy the
    nested set axioms so long as the union of the singletons is not the
    minimal element of the flag. The union of the singletons cannot be a
    minimal element of the flag, since $m+1$ is not contained in the
    singletons. Hence, this map is well-defined and is clearly of the inverse
    of the previous map. Thus, we have a bijection between the face lattices. 
		
    It remains to understand how the grading is mapped via this bijection
    using the notion of missing rank. Note that each face in the nested set
    complex is given by adding sets to $[m+1]$. In that sense, the nested set
    $\{[m+1]\}$ corresponds to the trivial $m$ dimensional face. Each face is
    attained by adding compatible sets to $[m+1]$ one at a time. A compatible
    set to a given system is either a singleton that appears in all sets
    containing $n+1$ or a set in the building set containing $m+1$ that is
    contained in or contains all the sets containing $m+1$. Under the
    bijection, adding a singleton corresponds to removing a single element
    from every set in the current chain, and adding a new set containing $m+1$
    corresponds to adding a set in the chain. Removing a single element from
    everything chain reduces the number of missing ranks by $1$ by decreasing
    the size of the maximal rank set without affecting any of the subsets
    considered for missing ranks. Similarly, adding a new set to the chain
    decreases the number of missing ranks by $1$ by filling up a rank.
    Therefore, both of these operations reduce the numbers of missing ranks by
    precisely $1$ and equivalently reduce the dimension by $1$. Hence, this
    bijection takes the dimension statistic to the missing ranks statistic,
    which finishes the proof.  
\end{proof}

\section{Paths on Base Polytopes}\label{sec:base_poly}

In the previous section, we showed that for a polymatroid $([n],f)$, the
linear embedding 
\[
    \{ (x,f(E)-\1(x)) : x \in P_f \}
\]
is the base polytope of an associated polymatroid $\ef$. Thus, $\1$-monotone
paths on $P_f$ correspond to $\1_{[n]}$-monotone paths on $B_{\ef} \subset
\R^{n+1}$. This suggests the investigation of (coherent) monotone paths on
polymatroid base polytopes or, equivalently, generalized permutahedra.  In this
section, we make some first observations on this very interesting but widely
unexplored subject.

We begin with the permutahedron $\Pi_{n-1}$. For a generic linear function $c$,
the oriented graph of the permutahedron is the Hasse diagram of the weak Bruhat
order of the symmetric group $\SymGrp_{n}$; cf.~\cite{BjornerBrenti}.  The
monotone paths are precisely maximal chains in the weak Bruhat order, which
correspond to reduced words for the longest element $(n,n-1,\dots,2,1)$. In this
language, the results of Edelmann--Greene~\cite{EdelmanGreene} can be
interpreted as a bijection between the set of monotone paths on the
permutahedron for a generic orientation and standard Young tableaux of staircase
shape. These monotone paths have also appeared in the literature under the name
of sorting networks, for which their random behavior has a remarkable
description~\cite{RandSort, ArchLim}.  Furthermore, there is a canonical method
of drawing a sorting network as a wiring diagram, which has appeared in the
study of cluster algebras for describing the cluster structure on the complete
flag variety \cite[Section 1.3]{ClusterBook}. Adjacency of monotone paths in the
Baues poset correspond to polygonal flips on the permutahedron, that is,
changing a monotone path along a $2$-dimensional face. These flips correspond
precisely to applying the Coxeter relations $s_is_{i+1}s_i = s_{i+1}s_is_{i+1}$
or $s_is_j = s_js_i$ for $j > i+1$ to a reduced expression.  Describing the
\emph{coherent} monotone paths on the permutahedron was left open in
Billera--Sturmfels~\cite{BSFiberPoly}. Coherent monotone paths for generic
orientations on the permutahedron have also appeared in the literature
previously under the name of allowable sequences or stretchable/geometrically
realizable sorting networks \cite{PattTheorem, GoodmanPollack,  sweeppolys}.
Theorem~1.3 of~\cite{PattTheorem} states that the proportion of coherent
monotone paths tends to $0$ for large $n$. 

Even for matroids, not all $c$-monotone paths on base polytopes will be
coherent. For example, the base polytope for the uniform matroid $U_{4,2}$ is
the hypersimplex $\Delta(4,2)$, which is linearly isomorphic to the octahedron
$C_3^\triangle = \conv(\pm e_1, \pm e_2, \pm e_3)$. Monotone paths on
cross-polytopes are studied in detail in~\cite{MPPofCP} and Theorem~1.1 in
this paper yields that there is always at least one incoherent monotone path
for every generic linear function.

In this section, we will focus on (coherent) monotone paths on $B_f$ with
respect to the special linear functions $\1_S(x) = \sum_{i \in S} x_i$ for $S
\subseteq [n]$. We start with the case of matroids.

\begin{theorem}
    Let $M$ be a matroid with ground set $E$ and $S \subseteq E$. Every
    $\1_S$-monotone path on the base polytope $B_M$ is coherent.
\end{theorem}

\begin{proof}
    Let $M$ be a matroid rank $r$ matroid on $E$ and $S \subseteq E$ a set of
    size $s$. Two bases $B,B'$ are adjacent on $B_M$ if and only if $B' = (B
    \setminus a) \cup b$ for $a \in B \setminus B'$ and $b \in B' \setminus
    B$. Moreover, $B'$ is a $\1_S$-improving neighbor of $B$ if and only if $a
    \not\in S$ and $b \in S$. In particular $\1_S(1_B - 1_{B'}) = 1$. 

    It follows that an $\1_S$-monotone path on $B_M$ is a sequence of bases
    $B_0, B_1,\dots, B_m$ such that $|B_0 \cap S|$ is minimal, $|B_m \cap S|$
    is maximal and $B_i = (B_{i-1} \setminus a_i) \cup b_i$ for $a_i \in
    B_{i-1} \setminus S$ and $b_i \in S$ for $i=1,\dots,m$. The elements
    $a_1,\dots,a_m,b_1,\dots,b_m$ are all distinct. For $0 < \epsilon \ll 1$
    and we can define $\wt \in \R^E$ by 
    \[
        \wt_e \ := \
        \begin{cases}
            1 & \text{ if } e \in B_0 \\
            \epsilon^i & \text{ if } e = a_i \\
            -\epsilon^i & \text{ if } e = b_i \\
            0 & \text{ otherwise.}
        \end{cases}
    \]
    Then $1_{B_0}$ is the unique maximizer of $\wt$ over $B_M \cap \{ \1_S(x) =
    |S \cap B_0|\}$ and Lemma~\ref{lem:cMP_local} asserts that $B_0,\dots,B_m$
    is the coherent path with respect to $\wt$.
\end{proof}

\begin{example}[Hypersimplices]
    Let $M = U_{n,k}$ be the uniform matroid of rank $k$ on $n$ elements and
    $S \subseteq [n]$ of cardinality $s$. The base polytope is the
    $(n,k)$-hypersimplex $\Delta(n,k)$ and the monotone path polytope
    $\MPPl{\1_S}{B_M} = \MPPl{\1_S}{\Delta(n,k)}$ is normally equivalent to
    \[
        \sum_{j = \max(0,k+s-n)}^{\min(k,s)} \Delta(s,j) \times
        \Delta(n-s,k-j) \, .
    \]
    Indeed, for every $k$-element subset $B \subseteq [n]$, we have
    $\max(0,k+s-n) \le|B \cap S| \le \min(k,s)$ and every value can be
    attained. For any $j$ in that range, $\Delta(n,k) \cap \{ \1_S(x) = j\}$ is
    the convex hull of $1_G + 1_H$, where $G \in \binom{S}{j}$ and $H \in
    \binom{[n] \setminus S}{k-j}$.

    If $s = k$ and $n = 2k$, then $\MPPl{\1_S}{\Delta(n,k)}$ is normally
    equivalent to $\Pi_k \times \Pi_k$.
\end{example}

It would be very interesting to further understand the combinatorics of
$\MPPl{\1_S}{B_M}$. Observe that for $|S| \ge 2$, the polytopes $B_M \cap
\{\1_S(x) = j\}$ for $j \in \Z$ are $0/1$-polytopes with edge directions $e_i -
e_j + e_k - e_l$.  Such polytopes were studied by Castillo and
Liu~\cite{CastilloFu22} in the context of nested braid fans. 

For $S = \{e\}$, where $e$ is not a loop or coloop, the monotone path polytope
$\MPPl{\1_S}{B_M}$ is normally equivalent to $B_{M \backslash e} + B_{M / e}$,
which is again a polymatroid base polytope. This holds in general.

\begin{prop}
    Let $(E,f)$ be a polymatroid and $S \subseteq E$ such that $|S| = 1$ or
    $|S| = |E| - 1$. Then $\MPPl{\1_S}{B_f}$ is a polymatroid base polytope.
\end{prop}
\begin{proof}
    Since $B_f \subset \{ \1(x) = f(E)\}$, the linear functions $\1_{\{e\}}$
    and $\1_{E \setminus \{e\}}$ induce the same monotone paths. We may thus
    assume that $S = \{e\}$.

    Let $\alpha \in \R$ such that $H_\alpha = \{ x \in \R^E : x_e = \alpha\}$
    meets $B_f$ in the relative interior. An edge of $B_f \cap H_\alpha$ is of
    the form $F \cap H_\alpha$, where $F \subset B_f$ is a face of dimension
    $2$ that meets $H_\alpha$ in its relative interior. Now, $F$ is itself the
    base polytope of a polymatroid, which is either the Cartesian product of
    two $1$-dimensional base polytopes or the base polytope of a polymatroid
    on three elements. In both cases, it follows that $F \cap H_\alpha =
    [u,v]$ and $u-v = \lambda (e_i - e_j)$ for some $i,j \in E \setminus
    \{e\}$. Since base polytopes are polytopes all whose edge directions are
    of the form $e_i-e_j$ for some $i,j \in E$, this proves the claim.
\end{proof}

We do not know if for a general polymatroid $(E,f)$ and $|S| = 1$, all
$\1_S$-monotone paths of $B_f$ are coherent nor what the corresponding
polymatroid is. 

\begin{example}[Monotone paths on the associahedron]
    Let $\Ass_{n-1} \subset \R^n$ be the Loday associahedron;
    cf.~Example~\ref{ex:associahedra}. Let $i \in [n]$. A binary tree $T$
    corresponding to a vertex $v$ of $\Ass_{n-1}$ maximizes $v_i$ (with value
    $i(n-i+1)$) if and only if $i$ is the root of $T$. It minimizes $v_i$
    (with value $1$) if and only if $i$ is a leaf. It is easy to see that
    for $S = \{i\}$, $\Ass_{n-1}^{-\1_S}$ is linearly isomorphic to
    $\Ass_{n-2}$. By removing
    the leaf $i$ from $T$ and relabelling every node $j > i$ to $j-1$, this
    yields a plane binary tree on $n-1$ nodes and every such tree arises
    uniquely this
    way. Two trees $T$ and $T'$ correspond to adjacent vertices of
    $\Ass_{n-1}$ if they differ by a \emph{rotation} of two adjacent nodes $x$
    and $y$. The tree $T'$ corresponds to a $\1_S$-improving neighbor of $T$
    iff the rotation decreased the distance of $i$ to the root. It follows
    that for every tree $T$ with $i$ a leaf, $T$ is the starting point of a
    unique monotone path. The nodes that are rotated along the path, readily
    yield a weight $\wt$ which certifies that the path is coherent. In
    particular, since the starting point of the path determines the whole
    path, this shows that the polytopes $\Ass_{n-1} \cap \{ x_i = \alpha\}$
    are weak Minkowski summands of $\Ass_{n-1} \cap \{ x_i = 1\} =
    \Ass_{n-1}^{-\1_S}$. Thus $\MPPl{\1_S}{\Ass_{n-1}}$ is normally equivalent
    to the associahedron $\Ass_{n-2}$.
\end{example}

We conclude the section with a discussion of $\1_S$-monotone paths on the
permutahedron.  Recall that a standard Young tableau (SYT) of shape $m \times
n$ is a rectangular array filled with numbers from $1,\dots,mn$ without
repetitions and such that rows and columns are increasing top-to-bottom and
left-to-right, respectively. Let $\SYT(m,n)$ denote the collection of all such
standard Young tableaux.

\begin{prop}\label{prop:lat_perm}
    For $1 \le k < n$ let $S = \{1,\dots,k\}$. The $\1_S$-monotone paths on
    the permutahedron $\Pi_{n-1}$ starting from $(1,2,\dots, n)$ are in
    bijection with standard Young tableaux of shape $k \times (n-k)$.
\end{prop} 

\begin{proof}
    A \Def{rectangular lattice permutation} of size $k \times (n-k)$ is a
    sequence $a_{1}, a_{2}, \dots, a_{k(n-k)} \in [k]$ such that
    \begin{enumerate}[\rm (i)]
        \item The number of occurrences of $j \in [k]$ is $n-k$, and 
        \item For any $1 \leq m \leq k(n-k)$, the number of occurrences of $i$
            in $a_1,\dots,a_m$ is at least as large as the number of
            occurrences of $i+1$. 
    \end{enumerate}
    For a rectangular lattice permutation, one associates a rectangular SYT by
    starting with an empty rectangular array and appending the number $k$ in
    row $a_k$.  Proposition 7.10.3 in~\cite{EC2} yields that this is a
    bijection from rectangular lattice permutations of size $k \times (n-k)$
    to rectangular SYT of shape $k \times (n-k)$. We give an explicit
    bijection between monotone paths and rectangular lattice permutations.

    Let $(1,2,\dots,n) = \sigma_0,\sigma_1,\dots,\sigma_M$ be a
    $\1_S$-monotone path on $\Pi_{n-1}$. For every $1 \le h \le M$, we have
    $\sigma_h - \sigma_{h-1} = e_i - e_j$ with $1 \le i \le k < j \le n$ and
    we  define $a_1,a_2,\dots,a_M$ by $a_h := k+1 - i$.

    Each step along the path will always increase the value of an element of
    the first $k$ coordinates by precisely $1$. Since the first $k$
    coordinates move from $(1,2,\dots, k)$ to $(n-k+1, \dots, n)$, the total
    length of the path is $M = k(n-k)$. It also shows that every $i =
    1,\dots,k$ occurs exactly $n-k$ times in $a_1,\dots,a_M$, which verifies
    (i). Moreover, $\sigma_h$ is a permutation for every $h$ and it can be
    seen that the first $k$ and the last $n-k$ entries of $\sigma_h$ are
    always increasing. This shows that (ii) is satisfied and hence
    $a_1,\dots,a_M$ is a lattice permutation. This also shows that $\sigma_M$
    is $(n-k+1,\dots,n,1,\dots,n-k)$.

    For a given rectangular lattice permutation $a_1,\dots,a_{k(n-k)}$, we
    define a sequence of permutations $\sigma_0,\dots,\sigma_{k(n-k)}$ as
    follows. We set $\sigma_0 := (1,\dots,n)$ and for $h \ge 1$, we define
    $\sigma_h$ by swapping the values $\sigma_{h-1}(k+1-a_h)$ and
    $\sigma_{h-1}(k+1-a_h)+1$. Since $1 \le k+1-a_h \le k$, this increases the
    values on the first $k$ coordinates. The sequence is well-defined since by
    (i), every coordinate is swapped $n-k$ times for a larger one. Moreover,
    condition (ii) ensures that the first $k$ coordinates are increasing,
    which implies that the sequence is a $\1_S$-monotone path.
\end{proof}

For $n=5$ and $k = 3$, a monotone path is given by 
\[
    \
    12345 \stackrel{e_3-e_4}{\longrightarrow} 
    12435 \stackrel{e_3-e_5}{\longrightarrow} 
    12534 \stackrel{e_2-e_4}{\longrightarrow} 
    13524 \stackrel{e_1-e_4}{\longrightarrow} 
    23514 \stackrel{e_2-e_5}{\longrightarrow} 
    24513 \stackrel{e_1-e_5}{\longrightarrow} 
    34512 \, .
\]
The corresponding rectangular lattice permutation is $1, 1, 2, 3, 2, 3$ and
the rectangular SYT is 
\[
    \begin{bmatrix}
        1 & 2 \\
        3 & 5 \\
        4 & 6
    \end{bmatrix}
\]

As observed by Postnikov in Example 10.4 of \cite{grasspolysubdivs},
$\SYT(k,n-k)$ is also in bijection with the \emph{longest} monotone paths on
the hypersimplex $\Delta(n,k)$ for a generic orientation. Thus, the monotone
paths on the permutahedron for these special orientations correspond exactly
to the longest monotone paths on hypersimplices for generic orientations. 

Not all rectangular SYT correspond to coherent monotone paths. Coherent
$\1_S$-monotone paths of $\Pi_{n-1}$ are related to \emph{realizable} SYT by
work of Mallows and Vanderbei~\cite{VanderbeiMallows}. For vectors $u \in
\R^m$ and $v \in \R^n$, the \Def{outer sum} or \Def{tropical rank-$1$ matrix}
is the matrix $u \oplus v \in \R^{m \times n}$ with $(u \oplus v)_{ij} := u_i
+ v_j$.  If $u_{1} < u_{2} < \cdots < u_{m}$ and $v_{1} < v_{2} < \cdots <
v_{n}$ are sufficiently generic, then all entries of $u \oplus v$ are distinct
and strictly increasing along rows and columns. Replacing every entry in $u
\oplus v$ with its rank yields a SYT of shape $m \times n$, that Mallows and
Vanderbei call \Def{realizable}. For example, for $u = (0,10,11)$ and $v =
(1,3)$, this yields the SYT above.  In~\cite{VanderbeiMallows}, they ask which
SYT are realizable. We write $\rSYT(m,n) \subseteq \SYT(m,n)$ for the
collection of realizable SYT. Realizable SYT are closely related to coherent
monotone paths.

\newcommand\SP{\mathrm{SP}}%
To prove Theorem~\ref{Ithm:rSYT}, we relate the monotone path polytopes to
another class of polytopes. For a finite set $T \subset \R^d$ the \Def{sweep
polytope} \cite{sweeppolys} is defined as 
\[
    \SP(T) \ := \ \tfrac{1}{2}\sum_{a,b \in T} [a-b,b-a] \, .
\]
This is a zonotope whose vertices record possible orderings of $T$ induced
by generic linear functions. Let us also write $Z(T) = \sum_{a \in T} [0,a]$
for the zonotope associated to $T$.

\begin{prop}[{\cite[Prop.~2.10]{sweeppolys}}] \label{prop:zonosweep}
    Let $T \subset \R^d$. For $c \in \R^d$ define $T' := \{
        \frac{a}{\inner{c}{a}} : a \in T, \inner{c}{a} \neq 0\}$ and $T'' := T
    \setminus T'$. Then the monotone path polytope $\MPPl{c}{Z(T)}$ is
    normally equivalent to $\SP(T') + Z(T'')$.
\end{prop}

In the case of the permutahedron and special orientations, we can be more
explicit.

\begin{prop} \label{thm:4perm}
    Let $n \ge 1$ and $\emptyset \neq S \subseteq [n]$. The monotone path polytope
    $\MPPl{\1_S}{\Pi_{n-1}}$ is normally equivalent to the sweep polytope
    $\SP(T')$ for $T' = \{ e_i - e_j : i \in S, j \not\in S \}$.
\end{prop}
\begin{proof}
    The permutahedron is normally equivalent to the zonotope $Z = Z(T)$ for $T
    = \{ e_i - e_j : i,j \in [n], i \neq j\}$. By
    Proposition~\ref{prop:zonosweep}, $\MPPl{\1_S}{Z(T)}$ is normally
    equivalent to $\SP(T') + Z(T'')$.  Note that $T'$ consists of all vectors
    $e_i - e_j$ for $i \in S$ and $j \in S^c := [n] \setminus S$. For $i,k \in
    S$ and $j \in S^c$ arbitrary, $e_i - e_k = (e_i - e_j) - (e_k - e_j)$ is a
    generator for $\SP(T')$ and hence $Z(T'')$ is a weak Minkowski summand of
    $\SP(T')$. This means that $\MPPl{\1_S}{Z(T)}$ is normally equivalent to
    $\SP(T')$.
\end{proof}

Note that for $S \subseteq [n]$ of size $k \ge 1$, the set $T' = \{ e_i - e_j
: i \in S, j \not\in S \} \subset \R^n$ is linearly isomorphic to $\{
(e_i,e_j) : i \in [k], j \in [n-k]\}$, that is, the vertices of the polytope
$\Delta_{k-1} \times \Delta_{n-k-1}$.

\begin{cor}\label{cor:prod_simplices}
    Let $n \ge 1$ and $S \subseteq [n]$ of size $k = |S| \ge 1$. Then
    $\MPPl{\1_S}{\Pi_{n-1}}$ is combinatorially equivalent to the sweep
    polytope of the product of simplices $\Delta_{k-1} \times \Delta_{n-k-1}$.
\end{cor}

Let us note Proposition~\ref{thm:4perm} also implies that
$\MPPl{\1_S}{\Pi_{n-1}}$ is normally equivalent to $\MPPl{\1_S}{Z(T')}$ and, by
Theorem 1.7 in \cite{PivotPoly}, to the max-slope pivot polytope of
$\Pi_{n-1}$ with respect to $\1_S$.

\begin{proof}[Proof of Theorem \ref{Ithm:rSYT}]
    Up to symmetry, we may assume that $S = \{1,\dots,k\}$. If
    $\sigma_0,\dots,\sigma_M$ is a $\1_S$-monotone path of $\Pi_{n-1}$, then
    the first $k$ coordinates of $\sigma_0$ are a permutation of $1,\dots,k$.
    Likewise, the last $n-k$ coordinates are a permutation of $k+1,\dots,n$.
    Up to symmetry, we can assume that $\sigma_0 = (1,\dots,n)$.
    Proposition~\ref{prop:lat_perm} now proves the first claim.

    As for the second claim, note that by Proposition~\ref{thm:4perm}, it
    suffices show that the vertices of $\SP(T')$ with $T' = \{ e_i - e_j : 1
    \le i \le k < j \le n \}$ are in bijection to $\SymGrp_{k} \times
    \SymGrp_{n-k} \times \rSYT(k,n-k)$. 

    Let $\wt \in \R^n$ be a generic linear function.  The segments $[e_i-e_j,
    e_j-e_i]$ for $i,j \in S$ or $i,j \in S^c$ are subsumed by $\SP(T')$.
    Hence, the sweep of $T'$ induced by $\wt$ totally orders $S$ and $S^c$.
    Without loss of generality, we may assume that the ordering is the natural
    ordering on $S$ and $S^c$. Let us write $\wt = (u,-v)$ with $u \in \R^k$
    and $v \in \R^{n-k}$. Then $\inner{\wt}{e_i - e_j} = u_i + v_j$. The sweep
    of $T'$ is thus determined by the ranks of $u \oplus v$ and hence
    determines a unique element in $\rSYT(k,n-k)$. Conversely, every element
    in $\rSYT(k,n-k)$ is determined by some $(u,-v)$ up to a total order on
    $S$ and $S^c$. This proves the claim.
\end{proof}

This perspective in terms of coherent $\1_S$-monotone paths provides an
alternative geometric perspective on realizable SYT. 

\begin{theorem}[{\cite{VanderbeiMallows}}]\label{thm:MV2}
    All rectangular standard Young Tableaux of shape $2 \times (n-2)$ are
    realizable. 
\end{theorem}

\begin{cor}
    For $S \subseteq [n]$ and $|S| = 2$, all $\1_S$-monotone paths on
    $\Pi_{n-1}$ are coherent. The number of such paths is $2(n-2)!C_{n-2}$,
    where $C_{k}$ denotes the $k$th Catalan number. 
\end{cor}

We give a short geometric proof of Theorem~\ref{thm:MV2}. To that end, we make
the observation that, since the sweep polytopes are zonotopes, the normal fan
of $\MPPl{\1_S}{\Pi_{n-1}}$ for $S = \{1,\dots,k\}$ is given by the
arrangement of hyperplanes 
\[
    \{ (x,y) \in \R^k \times \R^{n-k} : x_i + y_k = x_j + y_l \}
\]
for $i,j \in [m], k,l \in [n]$ with $i\neq j$ and $k \neq l$. Inspecting the
proof of Theorem~\ref{Ithm:rSYT}, we arrive at the following conclusion.

\begin{cor}
    For $m,n \ge 1$, realizable SYT of shape $m \times n$ are in bijection to
    the regions of the arrangement of hyperplanes $\{ (x,y) \in \R^m \times
    \R^n : x_i + y_k = x_j + y_l \}$ for $i,j \in [m], k,l \in [n]$ with
    $i\neq j$ and $k \neq l$ restricted to the cone $\{(x,y) : x_1 \le \cdots
    \le x_m, y_1 \le \dots \le y_n \}$.
\end{cor}

\begin{proof}[Proof of Theorem~\ref{thm:MV2}]
    Let $m = 2$.  the cone $\{(x,y) : x_1 \le  x_2,
    y_1 \le \dots \le y_n \}$ has a $2$-dimensional lineality space given by
    adding a constant to all coordinates of $x$ and, independently, to $y$.
    Since all the hyperplanes are linear, we may thus assume that $x_1 = 0$.
    We may also restrict to $x_2 = 1$ and count the number of regions in the
    cone $C = \{ y \in \R^n : y_1 \le \dots \le y_n \}$ induced by the
    hyperplanes 
    \[
        y_k - y_l \ = \ \pm 1
    \]
    for $k \neq l$. The cone $C$ is the fundamental cone for the braid
    arrangement and hyperplanes constitute the so called \Def{Catalan
    arrangement}.  The number of regions in $C$ is well-known to be $C_n$; see,
    for example,~\cite[Prop.~5.14]{Stanley-HA} or~\cite{Athanasiadis} for the
    connection to Shi arrangements.
\end{proof}

Mallows and Vanderbei also discuss realizability of general rectangular SYT
and show that the tableau
\[
    \footnotesize
    \begin{bmatrix}
        1 & 2 & 6\\
        3 & 5 & 7\\
        4 & 8 & 9
    \end{bmatrix}
\]
is not realizable.

\begin{cor}
    For $|S| \ge 3$ not all $\1_S$-monotone paths on $\Pi_{n-1}$ are
    coherent.
\end{cor}

All monotone paths being coherent on zonotopes in general is a strong
restriction. In~\cite{Edman} it is shown that for a generic objective function
all monotone paths being coherent implies that all cellular strings are
coherent and they provide a complete characterization of the cases in which this
arises. However, in this special case, the objective function $\1_{\{1,2\}}$
is not generic, so their tools do not apply. 

The combinatorics of the monotone path polytope for the permutahedron in other
cases remains complicated. However, it is surprisingly natural and connected
to applications through the motivation of Mallows and Vanderbei in
\cite{VanderbeiMallows}. We end on the open question of whether we can obtain
a more robust description of the $\1_{S}$-monotone path polytopes of large
classes of generalized permutahedra or the permutahedron itself. 

\begin{op}
    For fixed $m,n \ge 3$, determine the (number of) realizable SYT of shape
    $m \times n$. Equivalently, determine the (number of) coherent monotone
    paths of $\Pi_{n-1}$ for special directions $\1_S$ with $|S| \ge 3$.
\end{op}

\bibliographystyle{siam}
\bibliography{References.bib} 

\begin{thebibliography}{10}

\bibitem{AHK}
{\sc K.~Adiprasito, J.~Huh, and E.~Katz}, {\em Hodge theory for combinatorial
  geometries}, Ann. of Math. (2), 188 (2018), pp.~381--452.

\bibitem{PattTheorem}
{\sc O.~Angel, V.~Gorin, and A.~Holroyd}, {\em A pattern theorem for random
  sorting networks}, Electronic Journal of Probability, 17 (2012), pp.~1--16.

\bibitem{RandSort}
{\sc O.~Angel, A.~E. Holroyd, D.~Romik, and B.~Vir{\'a}g}, {\em Random sorting
  networks}, Advances in Mathematics, 215 (2007), pp.~839--868.

\bibitem{Doker}
{\sc F.~Ardila, C.~Benedetti, and J.~Doker}, {\em Matroid polytopes and their
  volumes}, Discrete Comput. Geom., 43 (2010), pp.~841--854.

\bibitem{Athanasiadis}
{\sc C.~A. Athanasiadis}, {\em Generalized {C}atalan numbers, {W}eyl groups and
  arrangements of hyperplanes}, Bull. London Math. Soc., 36 (2004),
  pp.~294--302.

\bibitem{CellString}
{\sc L.~Billera, M.~Kapranov, and B.~Sturmfels}, {\em Cellular strings on
  polytopes}, Proceedings of the American Mathematical Society, 122 (1994),
  pp.~549--555.

\bibitem{BSFiberPoly}
{\sc L.~Billera and B.~Sturmfels}, {\em Fiber polytopes}, Ann. of Math. (2),
  135 (1992), pp.~527--549.

\bibitem{Bjorner-Subspace}
{\sc A.~Bj\"{o}rner}, {\em Subspace arrangements}, in First {E}uropean
  {C}ongress of {M}athematics, {V}ol. {I} ({P}aris, 1992), vol.~119 of Progr.
  Math., Birkh\"{a}user, Basel, 1994, pp.~321--370.

\bibitem{BjornerBrenti}
{\sc A.~Bj\"{o}rner and F.~Brenti}, {\em Combinatorics of {C}oxeter groups},
  vol.~231 of Graduate Texts in Mathematics, Springer, New York, 2005.

\bibitem{MPPofCP}
{\sc A.~Black and J.~De~Loera}, {\em Monotone paths on cross-polytopes},
  Discrete \& Computational Geometry, 70 (2023), pp.~1245--1265.

\bibitem{PivotPoly}
{\sc A.~Black, J.~De~Loera, N.~L{\"u}tjeharms, and R.~Sanyal}, {\em The
  polyhedral geometry of pivot rules and monotone paths}, SIAM Journal on
  Applied Algebra and Geometry, 7 (2023), pp.~623--650.

\bibitem{flagmats}
{\sc A.~V. Borovik, I.~M. Gelfand, A.~Vince, and N.~White}, {\em The lattice of
  flats and its underlying flag matroid polytope}, Annals of Combinatorics, 1
  (1997), pp.~17--26.

\bibitem{CoxeterMat}
{\sc A.~V. Borovik, I.~M. Gelfand, and N.~White}, {\em Coxeter matroids},
  Springer, 2003.

\bibitem{TropicalFlagVarieties}
{\sc M.~Brandt, C.~Eur, and L.~Zhang}, {\em Tropical flag varieties}, Adv.
  Math., 384 (2021), pp.~Paper No. 107695, 41.

\bibitem{FlagMatAlgGeo}
{\sc A.~Cameron, R.~Dinu, M.~Micha{\l}ek, and T.~Seynnaeve}, {\em Flag
  matroids: Algebra and geometry}, in Interactions with Lattice Polytopes,
  A.~M. Kasprzyk and B.~Nill, eds., Springer, 2022, pp.~73--114.

\bibitem{lineuppolys}
{\sc F.~Castillo, J.-P. Labb{\'e}, J.~Liebert, A.~Padrol, E.~Philippe, and
  C.~Schilling}, {\em An effective solution to convex 1-body
  n-representability}, in Annales Henri Poincar{\'e}, Springer, 2023,
  pp.~1--81.

\bibitem{CastilloFu22}
{\sc F.~Castillo and F.~Liu}, {\em Deformation cones of nested braid fans},
  Int. Math. Res. Not.,  (2022), pp.~1973--2026.

\bibitem{ArchLim}
{\sc D.~Dauvergne}, {\em The archimedean limit of random sorting networks}, J.
  Amer. Math. Soc., online (2021).

\bibitem{EdelmanGreene}
{\sc P.~Edelman and C.~Greene}, {\em Balanced tableaux}, Advances in
  Mathematics, 63 (1987), pp.~42--99.

\bibitem{Edman}
{\sc R.~Edman, P.~Jiradilok, G.~Liu, and T.~McConville}, {\em Zonotopes whose
  cellular strings are all coherent}, European Journal of Combinatorics, 96
  (2021), p.~103352.

\bibitem{edmondsorig_old}
{\sc J.~Edmonds}, {\em Submodular functions, matroids, and certain polyhedra},
  Combinatorial Structures and Their Applications,  (1970), pp.~69--87.

\bibitem{edmondsorig}
\leavevmode\vrule height 2pt depth -1.6pt width 23pt, {\em Submodular
  functions, matroids, and certain polyhedra}, in Combinatorial
  optimization---{E}ureka, you shrink!, vol.~2570 of Lecture Notes in Comput.
  Sci., Springer, Berlin, 2003, pp.~11--26.

\bibitem{nestsets}
{\sc E.~Feichtner and B.~Sturmfels}, {\em Matroid polytopes, nested sets, and
  {B}ergman fans}, Portugaliae Mathematica (N.S.), 62 (2005), pp.~437 -- 468.

\bibitem{ClusterBook}
{\sc S.~Fomin, L.~Williams, and A.~Zelevinsky}, {\em Introduction to cluster
  algebras. chapters 1--3}.
\newblock August 2016,
  \href{https://arxiv.org/abs/1608.05735}{arXiv:1608.05735}.

\bibitem{Fujishige}
{\sc S.~Fujishige}, {\em Submodular functions and optimization}, vol.~58 of
  Annals of Discrete Mathematics, Elsevier B. V., Amsterdam, second~ed., 2005.

\bibitem{Fulkerson}
{\sc D.~R. Fulkerson}, {\em Anti-blocking polyhedra}, J. Combinatorial Theory
  Ser. B, 12 (1972), pp.~50--71.

\bibitem{GGMS}
{\sc I.~M. Gelfand, R.~M. Goresky, R.~D. MacPherson, and V.~V. Serganova}, {\em
  Combinatorial geometries, convex polyhedra, and {S}chubert cells}, Adv. in
  Math., 63 (1987), pp.~301--316.

\bibitem{GoodmanPollack}
{\sc J.~E. Goodman and R.~Pollack}, {\em Allowable sequences and order types in
  discrete and computational geometry}, in New trends in discrete and
  computational geometry, Springer, 1993, pp.~103--134.

\bibitem{grunbaum}
{\sc B.~Gr\"{u}nbaum}, {\em Convex polytopes}, vol.~221 of Graduate Texts in
  Mathematics, Springer-Verlag, New York, second~ed., 2003.

\bibitem{partperms}
{\sc D.~Heuer and J.~Striker}, {\em Partial permutation and alternating sign
  matrix polytopes}, SIAM Journal on Discrete Mathematics, 36 (2022),
  pp.~2863--2888.

\bibitem{Loho}
{\sc M.~Joswig, G.~Loho, D.~Luber, and J.~Alberto~Olarte}, {\em Generalized
  permutahedra and positive flag dressians}, International Mathematics Research
  Notices, 2023 (2023), pp.~16748--16777.

\bibitem{KSZ}
{\sc M.~M. Kapranov, B.~Sturmfels, and A.~V. Zelevinsky}, {\em Quotients of
  toric varieties}, Math. Ann., 290 (1991), pp.~643--655.

\bibitem{Knuth}
{\sc D.~E. Knuth}, {\em The art of computer programming. {V}olume 3},
  Addison-Wesley, Addison-Wesley Publishing Co., Reading, Mass.-London-Don
  Mills, Ont., 1973.

\bibitem{Loday}
{\sc J.-L. Loday}, {\em Realization of the {S}tasheff polytope}, Arch. Math.
  (Basel), 83 (2004), pp.~267--278.

\bibitem{TropGeom}
{\sc D.~Maclagan and B.~Sturmfels}, {\em Introduction to tropical geometry},
  vol.~161 of Graduate Studies in Mathematics, American Mathematical Society,
  Providence, RI, 2015.

\bibitem{VanderbeiMallows}
{\sc C.~Mallows and R.~J. Vanderbei}, {\em Which {Y}oung tableaux can represent
  an outer sum?}, Journal of Integer Sequences,  (2015).

\bibitem{Hypersimps}
{\sc S.~Manecke, R.~Sanyal, and J.~So}, {\em S-hypersimplices, pulling
  triangulations, and monotone paths}, The Electronic Journal of Combinatorics,
  27 (2020).

\bibitem{MillerSturmfels}
{\sc E.~Miller and B.~Sturmfels}, {\em Combinatorial commutative algebra},
  vol.~227 of Graduate Texts in Mathematics, Springer-Verlag, New York, 2005.

\bibitem{Oxley}
{\sc J.~Oxley}, {\em Matroid theory}, vol.~21 of Oxford Graduate Texts in
  Mathematics, Oxford University Press, Oxford, second~ed., 2011.

\bibitem{sweeppolys}
{\sc A.~Padrol and E.~Philippe}, {\em Sweeps, polytopes, oriented matroids, and
  allowable graphs of permutations}.
\newblock February 2021,
  \href{https://arxiv.org/abs/2102.06134}{arXiv:2102.06134}.

\bibitem{GenPermOrig}
{\sc A.~Postnikov}, {\em Permutohedra, associahedra, and beyond}, International
  Mathematics Research Notices, 2009 (2009), pp.~1026--1106.

\bibitem{grasspolysubdivs}
\leavevmode\vrule height 2pt depth -1.6pt width 23pt, {\em Positive
  {G}rassmannians and polyhedral subdivisions}, in Proceedings of the
  International Congress of Mathematicians, vol.~3, 2018, pp.~3167 -- 3196.

\bibitem{oeis}
{\sc N.~J.~A. Sloane}, {\em The on-line encyclopedia of integer sequences},
  published electronically at \url{https://oeis.org},  (July 2022).

\bibitem{Stanley-Ordered}
{\sc R.~P. Stanley}, {\em Ordered structures and partitions}, Memoirs of the
  American Mathematical Society, No. 119, American Mathematical Society,
  Providence, R.I., 1972.

\bibitem{EC2}
\leavevmode\vrule height 2pt depth -1.6pt width 23pt, {\em Enumerative
  combinatorics. {V}ol. 2}, vol.~62 of Cambridge Studies in Advanced
  Mathematics, Cambridge University Press, Cambridge, 1999.

\bibitem{Stanley-HA}
\leavevmode\vrule height 2pt depth -1.6pt width 23pt, {\em An introduction to
  hyperplane arrangements}, in Geometric combinatorics, vol.~13 of IAS/Park
  City Math. Ser., Amer. Math. Soc., Providence, RI, 2007, pp.~389--496.

\bibitem{sage}
{\sc W.~Stein et~al.}, {\em {S}age {M}athematics {S}oftware ({V}ersion 9.2)},
  The Sage Development Team, 2020.
\newblock {\tt http://www.sagemath.org}.

\bibitem{zieglerbook}
{\sc G.~M. Ziegler}, {\em Lectures on polytopes}, vol.~152, Springer Science \&
  Business Media, 2012.

\end{thebibliography}

\end{document}